\documentclass[aip,jmp, amssymb, amsmath]{revtex4-1}
\usepackage{graphicx}
\usepackage{amsfonts}
\usepackage{amsthm}
\usepackage{fancyhdr}
\usepackage{mathtools}
\usepackage{placeins}
\usepackage{listings}
\usepackage{caption}
\usepackage{subfigure}
\usepackage{xcolor}
\usepackage{comment}
\usepackage{bm}
\usepackage{hyperref}
\usepackage{bigints}
\usepackage{microtype}

\newtheorem{remark}{Remark}
\newtheorem{prop}{Proposition}
\newtheorem{thm}{Theorem}
\newtheorem{lemma}{Lemma}

\newtheorem{example}{Example}

\DeclareMathOperator*{\Tr}{Tr}

\allowdisplaybreaks

\begin{document}
\title{Some new results on relative entropy production, time reversal, and optimal control of time-inhomogeneous diffusion processes}

\date{\today}

\author{Wei Zhang}%
 \email{wei.zhang@fu-berlin.de}
\affiliation{ 
Zuse Institute Berlin, Takustrasse 7, 14195 Berlin, Germany.
%\\This line break forced with \textbackslash\textbackslash
}%

\begin{abstract}
  This paper studies time-inhomogeneous diffusion processes, including both
  Brownian dynamics and Langevin dynamics.  We derive upper bounds of the
  relative entropy production for a time-inhomogeneous process with respect to
  the transient invariant probability measures.  We also study the time
  reversal of the reverse process in Crooks' fluctuation theorem.  We show
  that the time reversal of the reverse process coincides with the optimally
  controlled forward process that leads to zero variance importance sampling
  estimator for free energy calculations based on Jarzynski's equality.
\end{abstract}


\maketitle

%AMS  60J60, 53C17
\section{Introduction}
\label{sec-intro}
In recent years, there has been growing research interest in understanding 
time-inhomogeneous systems in various research fields, such as
nonequilibrium physics,~\cite{crooks-path-ensemble-pre2000,fluct-thm-2012,generalized-fluct-thm-feedback}
molecular dynamics,~\cite{dellago-hummer-2014}
complex networks,~\cite{evolving-graphs-higham-2010,HOLME201297} and
biology.~\cite{bioattractors-dynamical-systems-theory-2014}
In nonequilibrium physics, in particular, one theoretical focus in the study
of time-inhomogeneous systems has been the 
nonequilibrium work relations, which concern dynamical processes
that are out of equilibrium under nonequilibrium driving forces. Notably, 
Jarzynski's equality relates the free energy differences between two
equilibrium states to the nonequilibrium
work needed to drive the system from one state to another within finite
time,~\cite{jarzynski1997,Jarzynskia2008} while fluctuation relations reveal
the connections between the forward process and the corresponding reverse
nonequilibrium process.~\cite{crooks1999,crooks-path-ensemble-pre2000,Chetrite2008,generalized-fluct-thm-feedback,fluct-thm-2012,non-equilibrium-2018}

In this paper we study time-inhomogeneous nonequilibrium diffusion processes,
including both Brownian dynamics (or overdamped Langevin dynamics) and
Langevin dynamics, which are widely used in modeling the dynamics of particle
systems in different application areas.~\cite{van2011stochastic} Our goal is to provide mathematical analysis on several topics of time-inhomogeneous processes,
namely the relative entropy, time reversal, as well as a type of optimal control problem that is
related to Jarzynski's equality.~\cite{non-equilibrium-2018}
Since the results for Brownian dynamics and Langevin dynamics are similar, 
in the following we summarize our results for Brownian dynamics and then discuss the differences
in the case of Langevin dynamics.  

Given $C^\infty$-smooth functions $J:\mathbb{R}^n\times [0,T]\rightarrow \mathbb{R}^n$,
$V:\mathbb{R}^n\times [0,T]\rightarrow \mathbb{R}$ and~$\sigma: \mathbb{R}^n \times [0, T] \rightarrow
\mathbb{R}^{n \times m}$, where $T>0$ and $m \ge n$, we consider the Brownian dynamics (forward process) in $\mathbb{R}^n$ which satisfies the stochastic differential equation (SDE)
\begin{align}
  \begin{split}
  d x(s)  =& b(x(s),s)\,ds + \sqrt{2\beta^{-1}}
  \sigma(x(s),s)\,dw(s)\,,  \\
    =& \Big(J - \gamma \nabla V + \frac{1}{\beta} \nabla \cdot \gamma\Big)(x(s), s)\,ds + \sqrt{2\beta^{-1}} \sigma(x(s),s)\,dw(s)\,, \quad s \in [0,T]\,,
  \end{split}
\label{forward-dynamics-intro}
\end{align}
where $\beta > 0$, $\gamma=\sigma\sigma^\top$, $b=J - \gamma \nabla V + \frac{1}{\beta} \nabla \cdot \gamma$, $(w(s))_{s\in [0,T]}$ is an $m$-dimensional Brownian motion,
and $\nabla\cdot \gamma: \mathbb{R}^n\times [0,T]\rightarrow
\mathbb{R}^n$ denotes the vector whose components are $(\nabla\cdot \gamma)_i = \sum_{j=1}^n \frac{\partial \gamma_{ij}}{\partial
x_j}$ for $1 \le i\le n$.
Note that both the drift $b$ and the diffusion coefficient $\sigma$ 
are time-dependent. The matrix-valued function $\gamma: \mathbb{R}^n\times
[0,T]\rightarrow \mathbb{R}^{n\times n}$ is assumed to be uniformly positive definite.  
For each $s \in [0, T]$, we denote by $\nu_s$ the probability distribution of
$x(s)\in \mathbb{R}^n$, and we assume that the vector field $J$ satisfies 
\begin{align}
  \mbox{div} \big(J(x, s) \mathrm{e}^{-\beta V(x, s)}\big) =0\,, \quad \mbox{a.e.}\hspace{0.2cm} x \in \mathbb{R}^n\,.
  \label{div-j-zero}
\end{align}
Under the assumption \eqref{div-j-zero} and certain technical conditions on
$V$ and $\sigma$, the process (for fixed $s$)
\begin{align}
  d \widetilde{x}^s(t)  = b(\widetilde{x}^s(t),s)\,dt + \sqrt{2\beta^{-1}}
  \sigma(\widetilde{x}^s(t),s)\,dw(t)\,, \quad t \ge 0
\label{forward-dynamics-intro-fixed-s}
\end{align}
is ergodic with respect to the invariant measure $\nu^\infty_s$ , given by 
\begin{align}
  d\nu^\infty_s(x) = \frac{1}{Z(s)} \mathrm{e}^{-\beta V(x,s)}\,dx\,, \quad
  \mbox{where} ~~ Z(s) = \int_{\mathbb{R}^n} \mathrm{e}^{-\beta V(x,s)}\,dx \,.
  \label{invariant-mu}
\end{align}
We call $\nu^\infty_s$ the transient invariant measure of \eqref{forward-dynamics-intro} at time $s$.
Note that the inclusion of the term $\frac{1}{\beta}\nabla \cdot \gamma$
in \eqref{forward-dynamics-intro} and the assumption \eqref{div-j-zero} on $J$ guarantee that the invariant measure 
of \eqref{forward-dynamics-intro-fixed-s} is the Gibbs measure \eqref{invariant-mu}. 
In fact, the vector field $J$ corresponds to the anti-symmetric part of the
generator of \eqref{forward-dynamics-intro-fixed-s} (see \eqref{l-forward} in Section~\ref{subsec-notations-brownian}).

Our first result for Brownian dynamics \eqref{forward-dynamics-intro}
(Theorem~\ref{thm-entropy-estimate-brownian}) concerns upper bound estimates of the relative entropy of $\nu_s$ with respect to $\nu^\infty_s$, defined as 
\begin{align}
  \mathcal{R}^{\mathrm{Bro}}(s) := D_{KL}(\nu_s\,\|\,\nu^\infty_s) = \int_{\mathbb{R}^n} \ln
  \frac{d\nu_s}{d\nu^\infty_s}\, d\nu_s\,, \label{entropy-sr-intro}
\end{align}
where $D_{KL}$ denotes the Kullback-Leibler divergence between two probability measures.
The derivation of the upper bounds of $\mathcal{R}^{\mathrm{Bro}}(s)$ relies on the formula of
the relative entropy production rate (Proposition~\ref{prop-production-rate-of-entropy}):
\begin{align}
    \frac{d \mathcal{R}^{\mathrm{Bro}}(s)}{ds}  =& -\beta \int_{\mathbb{R}^n} \frac{\partial V}{\partial s} d\nu^{\infty}_s + \beta \int_{\mathbb{R}^n} \frac{\partial
    V}{\partial s}\, d\nu_s - 
  \frac{1}{\beta} \int_{\mathbb{R}^n}  \Big|\sigma^\top\nabla
  \left(\ln\frac{d\nu_s}{d\nu^{\infty}_s}\right)\Big|^2 \, d\nu_s\,, 
    \label{formula-dr-intro}
\end{align}
as well as the logarithmic Sobolev inequality for $\nu^\infty_s$.
We point out that \eqref{formula-dr-intro} generalizes the results in
literature for time-homogeneous
processes~\cite{origin-of-entropy-gehao,hongqian-decomp} with respect to a fixed invariant measure to time-inhomogeneous processes
\eqref{forward-dynamics-intro} with respect to the transient invariant measures $\nu^\infty_s$.

Our second result for Brownian dynamics \eqref{forward-dynamics-intro} (Theorem~\ref{thm-connection-brownian})
is about the connection between the time reversal $(x^{R,-}(s))_{s\in[0,T]}$,
where $x^{R,-}(s)=x^R(T-s)$ for $s\in[0,T]$, of the so-called reverse process $(x^R(s))_{s\in[0,T]}$ of
\eqref{forward-dynamics-intro}, which satisfies~\cite{crooks-path-ensemble-pre2000,Jarzynskia2008} 
\begin{align}
  d x^R(s)  = 
  \Big(-J - \gamma \nabla V + \frac{1}{\beta} \nabla \cdot \gamma\Big)(x^R(s), T-s)\,ds 
  + \sqrt{2\beta^{-1}} \sigma(x^R(s),T-s)\,dw(s)\,,
\label{backward-dynamics-intro}
\end{align}
for $s \in [0,T]$, and the stochastic optimal control problem

\begin{align}
  \inf_{(u_s)_{s\in[0,T]}} \mathbf{E}\left(
  \int_{0}^{T} \frac{\partial V}{\partial s}(x^u(s), s)\,ds + \frac{1}{4} \int_0^T |u_s|^2\,ds\,\middle|\,x^u(0) = x\right)\,,\quad x \in \mathbb{R}^n
  \label{opt-control-problem-U-intro}
\end{align}
of the controlled forward process 
\begin{align}
  d x^u(s)  =  b(x^u(s), s)\,ds + \sigma(x^u(s), s)\, u_s\, ds + \sqrt{2\beta^{-1}} \sigma(x^u(s),s)\,dw(s) \,,
\label{dynamics-1-u-intro}
\end{align}
where $u_s \in \mathbb{R}^m$, $0 \le s \le T$, is called the feedback control force.
Note that, comparing to the forward process \eqref{forward-dynamics-intro},
there is a change of sign in front of the term $J$ in \eqref{backward-dynamics-intro}.
Concerning the optimal control problem \eqref{opt-control-problem-U-intro}--\eqref{dynamics-1-u-intro}, it is known that there exists a function $u^*: \mathbb{R}^n\times [0, T]\rightarrow \mathbb{R}^m$, such that $u^*_s=u^*(x,s)$ is the optimal control of \eqref{opt-control-problem-U-intro}--\eqref{dynamics-1-u-intro}, when the
state of the process is $x\in \mathbb{R}^n$ at time $s\in [0,T]$ (see Ref.~\onlinecite[Chapter III and Chapter IV]{fleming2006}).
Weights other than $1/4$  in front of the quadratic term in \eqref{opt-control-problem-U-intro}
can be considered as well, but one needs to rescale the coefficient (i.e.\
$1$) in the control term in \eqref{dynamics-1-u-intro} accordingly,
so that the optimally controlled process $(x^{u^*}(s))_{s\in[0,T]}$ remains
the same. We show that the time reversal $(x^{R,-}(s))_{s\in[0,T]}$ coincides with the optimally
controlled process $(x^{u^*}(s))_{s\in[0,T]}$, i.e.\ they have the same law on
the path space $C([0,T], \mathbb{R}^n)$, provided that they start from the same initial distribution.
This result is interesting as it relates processes in two different contexts,
i.e.\ one in the celebrated Crooks' fluctuation
theorem,~\cite{crooks1999,crooks-path-ensemble-pre2000} and one in the study
of optimal Monte Carlo estimators for free energy calculations based on Jarzynski's equality.~\cite{non-equilibrium-2018}

Similar results are obtained for time-inhomogeneous Langevin dynamics,
namely the relative entropy estimate (Theorem~\ref{thm-entropy-langevin})
as well as the connection between the time reversal of
reverse process and certain optimally controlled forward process (Theorem~\ref{thm-connection-langevin}).  In the
following, we briefly discuss the main difference in the estimate of relative entropy for Langevin dynamics. Consider the time-inhomogeneous Langevin dynamics 
\begin{align}
  \begin{split}
    dq(s) =&\, p(s)\,ds \\
    dp(s) =&\, -\nabla_q V(q(s),s)\,ds - \xi p(s)\,ds + \sqrt{2\beta^{-1}\xi} \,dw(s)
  \end{split}
  \label{langevin-eqn-forward-case-1-intro}
\end{align}
with the corresponding Hamiltonian $H(q,p,s)=V(q,s)+|p|^2/2$ for $q,p\in \mathbb{R}^n$, where $\xi>0$ is
a positive constant and $\nabla_q$ denotes the gradient with respect to $q$.
Denote by~$\pi_s$ the probability measure of $(q(s), p(s)) \in
\mathbb{R}^n\times \mathbb{R}^n$ at time $s$ and by $\pi^\infty_s$ the transient invariant
measure at time $s$ (similarly defined as $\nu^\infty_s$ \eqref{invariant-mu}
in the case of Brownian dynamics; see \eqref{mu-s-langevin} in Section~\ref{subsec-forward-backward-langevin}). Due to the hypoelliptic structure of 
\eqref{langevin-eqn-forward-case-1-intro}, difficulties arise when estimating the upper bound of the relative entropy, which is defined as
\begin{align}
  \mathcal{R}^{\mathrm{Lan}}(s) := D_{KL}(\pi_s\,\|\,\pi^\infty_s) 
  = \int_{\mathbb{R}^n\times\mathbb{R}^n} \, \ln
  \frac{d\pi_s}{d\pi^{\infty}_s}\,  d\pi_s\,.
  \label{entropy-sr-langevin-intro}
\end{align}
In fact, we will derive the formula of the relative entropy production rate
(Proposition~\ref{prop-production-rate-of-entropy-langevin}): 
\begin{align}
  \begin{split}
    \frac{d \mathcal{R}^{\mathrm{Lan}}(s)}{ds}  = & -\beta
    \int_{\mathbb{R}^n\times \mathbb{R}^n} \frac{\partial H}{\partial s}\,
    d\pi^\infty_s + \beta \int_{\mathbb{R}^n\times \mathbb{R}^n} \frac{\partial
    H}{\partial s}\, d\pi_s - 
    \frac{\xi}{\beta} \int_{\mathbb{R}^n\times \mathbb{R}^n}
    \left|\nabla_p \Big(\ln\frac{d\pi_s}{d\pi^\infty_s}\Big)\right|^2 \, d\pi_s\,. 
  \end{split}
  \label{production-rate-of-entropy-langevin-intro}
\end{align}
Different from \eqref{formula-dr-intro} which contains the
full gradient, only the partial gradient~$\nabla_p$ with respect to momenta $p$ appears in the last integral of
\eqref{production-rate-of-entropy-langevin-intro}. This brings obstacles when
estimating the upper bound of $\mathcal{R}^{\mathrm{Lan}}(s)$ using logarithmic Sobolev inequality of $\pi^\infty_s$. To overcome this difficulty, we apply the hypocoercivity theory developed in Ref.~\onlinecite{villani2009hypocoercivity} and estimate $\mathcal{R}^{\mathrm{Lan}}(s)$ by considering a modified
functional. 

Let us next mention several previous work on related topics. 
Relative entropy (estimate) has been widely considered in the study of partial
differential equations,~\cite{carrillo2003,villani2009hypocoercivity} functional
inequalities,~\cite{ledoux2001concentration,villani2008optimal}
gradient flows,~\cite{ambrosio2005gradient} coarse-graining of
stochastic processes~\cite{effective_dynamics,Duong_2018} in mathematics community, as
well as in the study of nonequilibrium thermodynamics in physics community.~\cite{crooks1999,origin-of-entropy-gehao,hongqian-decomp} The convergence of Langevin dynamics towards equilibrium has
been studied using hypocoercivity theory~\cite{Talay2002StochasticHS,isotropic-Hypoellipticity-nier,villani2009hypocoercivity,DOLBEAULT2009-linear-relaxation,exponential-rate-schmeister-2012}
(see Refs.~\onlinecite{Mattingly2002,Bellet2006} for the approach using Lyapunov
techniques, and Ref.~\onlinecite{eberle2019} for coupling approach).
The weak convergence of a Langevin dynamics in the long-time small-friction limit 
was studied in Ref.~\onlinecite{hairer-Pavliotis-langevin2008}. The asymptotics in the
overdamped limit as well as numerical splitting schemes based on Langevin dynamics has
been analyzed in
Refs.~\onlinecite{lelievre_stoltz_2016,Leimkuhler-stoltz-2015}. Beyond
the equilibrium regime, by extending the entropy method,~\cite{arnold-markowich-toscani-etc-2001,entropy-method-ref}
Ref.~\onlinecite{achleitner-arnold-stuerzer} and Ref.~\onlinecite{arnald-carlen-ju-2008} studied the entropy decay of non-symmetric
Fokker-Planck type equations under a generalized Bakry-Emery condition. The
recent work Refs.~\onlinecite{bouin-homann-mouhot,Iacobucci2019} considered
the convergence rates of some nonequilibrium Langevin dynamics under external
forcing in a perturbative regime and established exponential convergence to
equilibrium with explicit rates. In other directions, we note that the reverse
process has drawn considerable attentions in nonequilibrium thermodynamics.~\cite{crooks1999,Chetrite2008,Jarzynskia2008} The optimal control problem \eqref{opt-control-problem-U-intro}--\eqref{dynamics-1-u-intro} related to Jarzynski's
equality has been considered in Ref.~\onlinecite{non-equilibrium-2018}.
Backward SDEs and controlled dynamics are also related to the so-called
Schr{\"o}dinger bridge problem, which has found applications in data assimilation~\cite{reich_2019} and the development of new Monte Carlo
schemes.~\cite{schroedinger-bridge-sampler}
Lastly, the time reversal of a diffusion process has been studied in Ref.~\onlinecite{haussmann1986}.

The current paper has the following novelties. First, concerning the results on
the upper bound of relative entropy, although this topic has been extensively studied, it seems that time-inhomogeneous processes as well as the associated relative entropy with respect to transient invariant
probability measures are less studied in literature.
Let us also emphasize that, different from the previous work Refs.~\onlinecite{hairer-Pavliotis-langevin2008,Iacobucci2019} which focused on 
the convergence (rate) of the processes in terms of certain parameters, our
aim is to derive the equations of the relative entropy production rate in
time-inhomogeneous case (see \eqref{formula-dr-intro} and
\eqref{production-rate-of-entropy-langevin-intro}) and to demonstrate their usefulness.
Second, concerning the results on the connection between the time reversal of
the reverse process and the optimally controlled process, the results are new to the best of the author's knowledge. 
At the technical level, the SDE of the time reversal
process obtained in mathematics community~\cite{haussmann1986} 
seems to be less known to the physics community, and the time reversal is
often studied by considering time discretization. In this paper, we 
prove the connection between two different processes by making use of this SDE. 

The paper is organized as follows. In Section~\ref{sec-overdamped}, we study time-inhomogeneous Brownian dynamics and obtain the aforementioned results. Langevin dynamics and the
corresponding results are considered in Section~\ref{sec-langevin}. 
The optimal control problems and their relations to Jarzynski's equality are
discussed in Appendix~\ref{app-sec-control-jarzynski}. Finally,
the proof of relative entropy estimate for Langevin dynamics is given in Appendix~\ref{app-sec-entropy-estimate-langevin}.

Before concluding this section, we introduce notation and recall
some useful inequalities related to probability measures. 
We denote by $I_n \in \mathbb{R}^{n\times n}$ the identity matrix of order $n$. 
For a vector $v\in \mathbb{R}^n$, $|v|$ is the Euclidean norm of $v$.
Let $A\in \mathbb{R}^{n\times n}$ be an $n$ by $n$ matrix.
The $2$-matrix norm of $A$, denoted by $\|A\|_2$, is the smallest nonnegative real number such that $|Av| \le \|A\|_2 |v|$ for all $v \in \mathbb{R}^n$.
  The operator $A:\nabla^2$ denotes the contraction between $A$ and the
Hessian operator $\nabla^2$, i.e.\ $A:\nabla^2 f = \sum\limits_{1 \le i,j\le n} A_{ij}\frac{\partial^2
f}{\partial x_i\partial x_j}$, for a $C^2$-smooth function $f: \mathbb{R}^n \rightarrow \mathbb{R}$.
We denote by $\mathbf{P}$ and $\mathbf{E}$ the probability and the
mathematical expectation on $\mathbb{R}^n$ (or on path space), respectively. 
Given a function $f:
\mathbb{R}^n\times [0, T]\rightarrow \mathbb{R}$ and time $t \in [0,T]$, 
$f(\cdot, t)$ is a function 
mapping from $\mathbb{R}^n$ to $\mathbb{R}$. We assume that a matrix-valued
function $\sigma: \mathbb{R}^n \times [0, T] \rightarrow \mathbb{R}^{n \times
m}$ is given such that $\gamma(x,s) = (\sigma\sigma^\top)(x,s)$ is
uniformly positive definite matrix, for all $(x,s)\in \mathbb{R}^n\times
[0,T]$. Precisely, we assume that there exists a positive constant $\gamma^->0$, such that 
\begin{align}
  v^\top \gamma(x,s) v  \ge \gamma^- |v|^2 \,, \quad \forall~v, x \in
  \mathbb{R}^n~~\mbox{and}~~\forall~s\in [0,T].
  \label{elliptic-gamma}
\end{align}
Note that this requires in particular that $m \ge n$ and $\sigma$ has full rank $n$.
For the matrix-valued function $\gamma: \mathbb{R}^n\times [0,T]\rightarrow \mathbb{R}^{n\times n}$ whose entries are
$\gamma_{ij}$, where $1 \le i,j\le n$, we denote by $\nabla\cdot \gamma$ 
the function mapping from $\mathbb{R}^n\times [0,T]$ to $\mathbb{R}^{n}$,
whose $i$th component is $(\nabla \cdot \gamma)_i = \sum_{j=1}^n \frac{\partial \gamma_{ij}}{\partial
  x_j}$, where $1 \le i \le n$. 
Next, let us recall some definitions and inequalities related to
probability measures.~\cite{villani2008optimal} 
Since we will apply the following results to measures in Euclidean spaces of different dimensions,
we consider a general Euclidean space $\mathbb{R}^d$, where $d\ge 1$. 
Given two probability measures $\mu$ and~$\nu$ on $\mathbb{R}^d$,  
their total variation is 
\begin{equation}
  \|\mu - \nu\|_{TV} = 2\inf \mathbf{P} [X\neq Y] 
  \label{def-tv}
\end{equation}
where the infimum is over all couplings $(X,Y)$ of $(\mu, \nu)$. Let 
\begin{equation*}
  \mathbf{B}^{\infty}(\mathbb{R}^d) =\{f:\mathbb{R}^d\rightarrow \mathbb{R}\,|\, f\,
  \mbox{is measurable and everywhere bounded}\}
\end{equation*}
be the Banach space endowed with the supremum norm $\|f\|_\infty := \sup_{x\in
\mathbb{R}^d} |f(x)|$, where $f \in \mathbf{B}^{\infty}(\mathbb{R}^d)$.
Similar definition can be made for vector-valued functions (i.e.\ with the supremum of $\ell^2$ vector norm).
For all $f\in \mathbf{B}^\infty(\mathbb{R}^d)$, 
\eqref{def-tv} implies 
\begin{equation}
  \left|\int_{\mathbb{R}^d} f\, d\mu - \int_{\mathbb{R}^d} f\,d\nu\right| \le
  \|f\|_\infty \|\mu-\nu\|_{TV}\,.
  \label{integration-bounded-by-tv}
\end{equation}
We write $\nu \ll \mu$ whenever $\nu$ is absolutely continuous with respect to $\mu$. 
The relative entropy of $\nu$ with respect to $\mu$ (also called the
Kullback-Leibler divergence from $\mu$ to $\nu$), where $\nu \ll \mu$, is 
\begin{equation}
D_{KL}(\nu\,\|\,\mu) = \int_{\mathbb{R}^d} \ln \frac{d\nu}{d\mu}\, d\nu\,.
  \label{kl-divergence}
\end{equation}
Csisz\'ar-Kullback-Pinsker inequality~(see Ref.~\onlinecite[Remark 22.12]{villani2008optimal}) states 
  \begin{align}
     \|\mu - \nu\|_{TV} \le  \sqrt{2D_{KL}(\nu\,\|\,\mu)} \,.
     \label{csiszar-kullback-pinsker}
  \end{align}
We also introduce the Fisher information
\begin{equation}
  \mathcal{I}(\nu\,\|\,\mu) 
  = \int_{\mathbb{R}^d} \Big|\nabla \left(\ln\frac{d\nu}{d\mu}\right)\Big|^2 \, d\nu\,. 
  \label{fisher-info}
\end{equation}
The probability measure $\mu$ satisfies the logarithmic Sobolev inequality
with constant $\kappa>0$~(see Ref.~\onlinecite[Definition 21.1]{villani2008optimal}
and Ref.~\onlinecite[Chapter 5]{ledoux2001concentration}), if 
\begin{align}
  D_{KL}(\nu\,\|\,\mu) \le \frac{1}{2\kappa} \mathcal{I}(\nu\,\|\,\mu)\,,
  \label{lsi}
\end{align}
for all probability measures $\nu$ such that $\nu \ll \mu$ and
$\mathcal{I}(\nu\,\|\,\mu)$ is finite.
For $p\in [1,+\infty)$, we denote by $\mathbf{W}_p(\mu, \nu)$ the Wasserstein distance of
order $p$ between $\mu$ and $\nu$, defined by 
(see Ref.~\onlinecite[Definition 6.1]{villani2008optimal} and Ref.~\onlinecite[Chapter 7]{ambrosio2005gradient}) 
\begin{equation}
  \mathbf{W}_p(\mu, \nu)=\inf\left\{(\mathbf{E}|X-Y|^p)^{\frac{1}{p}},~ \mbox{law}(X)=\mu, ~ \mbox{law}(Y)=\nu\right\},
\end{equation}
where the infimum is over all couplings $(X,Y)$ of $(\mu, \nu)$.
In this paper we use the Wasserstein distances $\mathbf{W}_1$ and
$\mathbf{W}_2$, which satisfy (by H\"{o}lder's inequality)
\begin{equation}
  \mathbf{W}_1(\mu,\nu) \le \mathbf{W}_2(\mu, \nu)\,,
  \label{w1-bounded-by-w2}
\end{equation}
for all probability measures $\mu, \nu$, whenever $\mathbf{W}_2(\mu, \nu) < \infty$.
Furthermore, for all Lipschitz functions $f$ on $\mathbb{R}^d$ with Lipschitz constant $\|f\|_{\textrm{Lip}}$, 
assuming $f$ is integrable with respect to both $\mu$ and $\nu$, we have 
\begin{align}
  \left|\int_{\mathbb{R}^d} f\, d\mu - \int_{\mathbb{R}^d} f\,d\nu\right| \le
  \|f\|_{\textrm{Lip}}\, \mathbf{W}_1(\mu,\nu)\,.
  \label{Lipschitz-w1}
\end{align}
When $\mu$ satisfies the logarithmic Sobolev inequality with constant $\kappa>0$, $\mu$ also
satisfies the Talagrand inequality with the same constant
$\kappa$,~\cite{OTTO-villani-on-talagrand} i.e.\ for
any probability measure $\nu \ll \mu$, 
      \begin{align}
	\mathbf{W}_2(\mu, \nu) \le \sqrt{\frac{2}{\kappa} D_{KL}(\nu\,\|\,\mu)}\,.
	\label{talagrand-ineq}
      \end{align}
\section{Time-inhomogeneous Brownian dynamics}
\label{sec-overdamped}
In this section, we consider time-inhomogeneous Brownian dynamics in~$\mathbb{R}^n$.
After introducing the forward and reverse processes as well as related
quantities in Section~\ref{subsec-notations-brownian}, we derive the 
entropy production rate formula and the upper bounds of relative entropy in Section~\ref{subsec-entropy-brownian}. Finally, in Section~\ref{subsec-crooks-control}, we establish the connection between the time
reversal of the reverse process and certain optimally controlled forward process.
\subsection{Forward and reverse processes}
\label{subsec-notations-brownian}
First, we introduce the forward process which will be studied in Section~\ref{subsec-entropy-brownian}.
Let $V:\mathbb{R}^n \times [0,T] \rightarrow \mathbb{R}$ be a time-dependent
$C^\infty$-smooth potential function in the time interval $[0,T]$, where $T>0$. 
The vector field $J:\mathbb{R}^{n} \times [0,T]\rightarrow \mathbb{R}^n$
satisfies the condition \eqref{div-j-zero}. Let $\beta > 0$ be a constant and $(w(s))_{s\in[0,T]}$ be an $m$-dimensional Brownian motion. We are interested in the time-inhomogeneous forward process in $\mathbb{R}^n$,
which satisfies the SDE
\begin{align}
  \begin{split}
  d x(s)     =\, & b(x(s),s)\,ds + \sqrt{2\beta^{-1}} \sigma(x(s),s)\,dw(s)\\
    =\,& \Big(J - \gamma \nabla V + \frac{1}{\beta} \nabla \cdot \gamma\Big)(x(s),
  s)\,ds + \sqrt{2\beta^{-1}} \sigma(x(s),s)\,dw(s)\,, \quad s \in [0,T]\,.
  \end{split}
\label{dynamics-1-q-vector}
\end{align}
In \eqref{dynamics-1-q-vector}, the drift vector $b: \mathbb{R}^n\times[0,T] \rightarrow \mathbb{R}^n$ is 
\begin{align}
  b = J - \gamma \nabla V + \frac{1}{\beta} \nabla \cdot \gamma\,,
  \label{drift-vector-b}
\end{align}
and $\gamma$ is related to $\sigma$ by $\gamma=\sigma\sigma^\top$ and satisfies the condition \eqref{elliptic-gamma}.
Following the terminology in Ref.~\onlinecite[Introduction]{markov-process-eberle}, 
we call the operator $\mathcal{L}_s$, defined by
\begin{align}
  \begin{split}
  \mathcal{L}_s f =& \Big[\big(J - \gamma \nabla V + \frac{1}{\beta} \nabla \cdot
  \gamma\big)(\cdot, s)\Big] \cdot \nabla f + \frac{1}{\beta} \gamma(\cdot,s):\nabla^2 f\\
    = & J \cdot \nabla f + \frac{\mathrm{e}^{\beta V}}{\beta} \mbox{div}
    (\mathrm{e}^{-\beta V} \gamma \nabla f)\,,
  \end{split}
  \label{l-forward}
\end{align}
for a test function $f: \mathbb{R}^n \rightarrow \mathbb{R}$, the generator of \eqref{dynamics-1-q-vector} at
time $s$. For fixed $s \in [0, T]$, under the assumption \eqref{div-j-zero} on $J$ and proper conditions
on $V$ and $\gamma$ (see Refs.~\onlinecite{Mattingly2002,roberts-tweedie-1996}, for instance),
the process \eqref{forward-dynamics-intro-fixed-s}, whose generator is $\mathcal{L}_s$,
has a unique invariant distribution $\nu_{s}^{\infty}$ given by the Gibbs measure \eqref{invariant-mu}.
In this case, the second line of \eqref{l-forward} is a decomposition of
$\mathcal{L}_s$ into the anti-symmetric and symmetric parts in $L^2(\mathbb{R}^n, \nu_{s}^{\infty})$.
Denote by $\mathcal{L}^*_s$ the adjoint operator of $\mathcal{L}_s$ in $L^2(\mathbb{R}^n, \nu_{s}^{\infty})$ 
(with respect to the weighted inner product defined by $\nu_{s}^{\infty}$).
Integrating by parts, together with \eqref{div-j-zero}, \eqref{invariant-mu} and \eqref{l-forward}, gives
\begin{equation}
  \begin{aligned}
  \mathcal{L}^*_s f = & -J \cdot \nabla f + \frac{\mathrm{e}^{\beta V}}{\beta} \mbox{div}
    (\mathrm{e}^{-\beta V} \gamma \nabla f)\,,
  \end{aligned}
  \label{l-adjoint}
\end{equation}
which differs from $\mathcal{L}_s$ \eqref{l-forward} by a change of sign in the anti-symmetric part. 

For each $s \in [0, T]$, we denote by $\nu_s$ the probability distribution of
$x(s)\in \mathbb{R}^n$, and we assume that $\nu_s$ has a 
positive and $C^\infty$-smooth probability density
$\rho$ with respect to the Lebesgue measure (see
Remark~\ref{rmk-regularity-of-rho-brownian} below), i.e.\ $\rho>0$ and 
\begin{equation}
  d\nu_s(x) = \rho(x,s)\, dx\,, \quad \forall~ s \in [0,T]\,.
  \label{distribution-nu}
\end{equation}
It is well known that $\rho$ satisfies the Fokker-Planck equation (see
Ref.~\onlinecite[Chapter 4]{risken1996fokker}), which, with the adjoint operator 
$\mathcal{L}^*_s$ \eqref{l-adjoint}, reads
~\footnote{Here $\rho$ is the density of $\nu_s$ with respect to the Lebesgue
measure and satisfies Fokker-Planck equation $\frac{\partial \rho}{\partial s}
= \mathcal{L}^\dagger_s \rho$, where $\mathcal{L}^\dagger_s$ 
is the adjoint operator of $\mathcal{L}_s$ in $L^2(\mathbb{R}^n)$ with respect
to the Lebesgue measure. It is direct to check that 
$\mathcal{L}^\dagger_s f = \mathrm{e}^{-\beta V} \mathcal{L}^*_s
(\mathrm{e}^{\beta V}f)$, for a test function $f$}
\begin{align}
  \frac{\partial \rho}{\partial s} = \mathrm{e}^{-\beta V} \mathcal{L}^*_s (\mathrm{e}^{\beta V}\rho)\,. 
  \label{fokker-planck}
\end{align}
The free energy of the process \eqref{dynamics-1-q-vector} at time $s$ is
defined as (recall $Z(s)$ in \eqref{invariant-mu})
\begin{align}
  F(s) = -\beta^{-1} \ln Z(s)\,,\quad ~s\in[0,T]\,.
  \label{free-energy-f}
\end{align}

Next, we introduce the reverse process which will be the subject of Section~\ref{subsec-crooks-control}.
Corresponding to the forward process \eqref{dynamics-1-q-vector}, the reverse
process is defined by 
\begin{align}
  \begin{split}
  d x^{R}(s) =& \Big(-J - \gamma \nabla V + \frac{1}{\beta} \nabla \cdot
    \gamma\Big)(x^R(s), T-s)\,ds + \sqrt{2\beta^{-1}} \sigma(x^R(s),T-s)\,dw(s)\,, 
  \end{split}
\label{dynamics-1-q-vector-backward}
\end{align}
where $s \in [0,T]$ and $(w(s))_{s\in[0,T]}$ is another Brownian motion in $\mathbb{R}^m$.
Similar to \eqref{l-forward}, for the reverse process
\eqref{dynamics-1-q-vector-backward} and for each $s \in [0, T]$, we define the operator 
\begin{align}
  \mathcal{L}^R_s f = \Big[\big(-J - \gamma \nabla V + \frac{1}{\beta} \nabla \cdot
  \gamma\big)(\cdot, T-s)\Big] \cdot \nabla f + \frac{1}{\beta} \gamma(\cdot,T-s):\nabla^2 f\,,
  \label{l-r-s}
\end{align}
for a test function $f: \mathbb{R}^n \rightarrow \mathbb{R}$.
The following lemma summarizes the relations among the operators introduced above.
\begin{lemma}
  Given $f,g\in C^2(\mathbb{R}^n)$, for all $s\in [0,T]$, we have
  $\mathcal{L}^*_s = \mathcal{L}^R_{T-s}$, and 
\begin{align}
     \int_{\mathbb{R}^n} f\big(\mathcal{L}_s + \mathcal{L}_s^*\big)g
    \,d\nu^{\infty}_s = -\frac{2}{\beta} \int_{\mathbb{R}^n} (\gamma\nabla f)\cdot\nabla
    g\, d\nu^{\infty}_s\,.
  \label{l-lr-dual}
\end{align}
Furthermore, when $g$ is positive, we have
\begin{equation}
  \begin{aligned}
    \frac{\mathcal{L}_s g}{g} =&~ \mathcal{L}_s (\ln g) +\frac{1}{\beta}
    |\sigma^\top\nabla\ln g|^2\,,\\
    \frac{\mathcal{L}^*_s g}{g} =&~ \mathcal{L}^*_s (\ln g)
    +\frac{1}{\beta} |\sigma^\top\nabla\ln g|^2\,.
  \end{aligned}
  \label{lg-llogg}
\end{equation}
  \label{lemma-generator-property}
\end{lemma}
\begin{proof}
  The identity $\mathcal{L}^*_s = \mathcal{L}^R_{T-s}$ follows directly from
  \eqref{l-adjoint} and \eqref{l-r-s}.
  The identity \eqref{l-lr-dual} is obtained by summing up
  \eqref{l-forward} and \eqref{l-adjoint}, and integrating by parts.
  The identities in \eqref{lg-llogg} can be verified directly using
  \eqref{l-forward} and \eqref{l-r-s}, together with the relation
  $\gamma=\sigma\sigma^\top$. 
\end{proof}

\begin{remark}
  Under the condition~\eqref{elliptic-gamma}, the smoothness of the solution
  $\rho$ to \eqref{fokker-planck} follows from the standard theory for parabolic partial differential equations (see Ref.~\onlinecite[Chapter
  11]{lorenzi2006analytical} and Ref.~\onlinecite[Chapter 6]{Bogachev:2264101} for instance). 
The positivity of $\rho$ can be argued by representing it as certain ensemble average
  of the reverse process \eqref{dynamics-1-q-vector-backward} using
  Feynman-Kac formula. 
  See \eqref{solution-of-rho-reverse} in Section~\ref{subsec-crooks-control} 
 for the representation of the density of the reverse process starting from
  $\nu_T^\infty$, and also the proof of Ref.~\onlinecite[Theorem 3.3.6.1]{michel-pardoux-1990} for details. 
Results on the positivity of densities can be found in Ref.~\onlinecite[Chapter 8.3]{Bogachev:2264101} 
as well.
  \label{rmk-regularity-of-rho-brownian}
\end{remark}
Before concluding, we note that Lemma~\ref{lemma-generator-property} can be
used to give a concise derivation of the fluctuation relation between the
forward process~\eqref{dynamics-1-q-vector} and the reverse process~\eqref{dynamics-1-q-vector-backward}. Since this is not the main purpose of this paper, we omit the discussions and
refer to Refs.~\onlinecite{Chetrite2008,non-equilibrium-2018} for details.

\subsection{Relative entropy estimate}
\label{subsec-entropy-brownian}
For $s\in[0,T]$, recall that $\nu_s$ is the probability distribution of
$x(s)\in \mathbb{R}^n$ whose density is $\rho(\cdot, s)$ in \eqref{distribution-nu} and that
$\nu^\infty_s$ is defined in \eqref{invariant-mu}. In this section, we study the relative entropy of $\nu_s$ with respect to $\nu^\infty_s$, defined as 
\begin{align}
  \mathcal{R}^{\mathrm{Bro}}(s) := D_{KL}(\nu_s\,\|\,\nu^\infty_s) = \int_{\mathbb{R}^n} \ln
  \frac{d\nu_s}{d\nu^\infty_s}(x)\, d\nu_s(x)
  = \int_{\mathbb{R}^n} \rho(x,s)\, \ln \frac{d\nu_s}{d\nu^\infty_s}(x)\,  dx \,.
  \label{entropy-sr}
\end{align}
The main result of this section is Theorem~\ref{thm-entropy-estimate-brownian}, 
which provides upper bounds of $\mathcal{R}^{\mathrm{Bro}}(s)$. 
Before stating Theorem~\ref{thm-entropy-estimate-brownian}, we first present the following useful results.
\begin{lemma}
  For $s\in[0,T]$, we have 
\begin{align}
  \frac{\partial }{\partial s} \Big(\ln\frac{d\nu_s}{d\nu^\infty_s}\Big)
  =& 
  \beta \Big(\frac{\partial V}{\partial s}
  - \frac{d F}{ds} \Big)
   +
  \mathcal{L}_s^* \Big(\ln\frac{d\nu_s}{d\nu^\infty_s}\Big)+
  \frac{1}{\beta}\Big|\sigma^\top\nabla \left(\ln\frac{d\nu_s}{d\nu^\infty_s} \right)\Big|^2 \,.
  \label{log-pde-relative-density}
\end{align}
  \label{lemma-equation-of-log-density-nu-to-mu}
\end{lemma}
\begin{proof}
  From \eqref{distribution-nu} and \eqref{invariant-mu}, we have 
  \begin{equation}
    \frac{d\nu_s}{d\nu^{\infty}_s}(x) = \mathrm{e}^{\beta V(x,s)}
\rho(x,s) Z(s)\,, \quad x \in \mathbb{R}^n\,.
    \label{density-of-nu-wrt-mu}
  \end{equation}
  Using the Fokker-Planck equation \eqref{fokker-planck} and the free energy \eqref{free-energy-f}, we derive
  \begin{equation*}
\begin{aligned}
   \frac{\partial}{\partial s} \Big(\frac{d\nu_s}{d\nu^{\infty}_s}\Big) 
  =& \mathcal{L}^*_s \big(\mathrm{e}^{\beta V}\rho\big) Z(s)
  + \beta \frac{\partial V}{\partial s} \frac{d\nu_s}{d\nu^{\infty}_s} 
  +\mathrm{e}^{\beta V} \rho\,\frac{dZ(s)}{ds} \\
  =& \mathcal{L}^*_s \Big(\frac{d\nu_s}{d\nu^{\infty}_s}\Big)
  + \beta \frac{\partial V}{\partial s} \frac{d\nu_s}{d\nu^{\infty}_s} 
  -\beta  \mathrm{e}^{\beta V}\rho\, Z(s)\frac{dF}{ds} \\
  =& \beta\Big(\frac{\partial V}{\partial s} - \frac{dF}{ds}
  \Big)\frac{d\nu_s}{d\nu^{\infty}_s} + \mathcal{L}^*_s \Big(\frac{d\nu_s}{d\nu^{\infty}_s}\Big)\,.
\end{aligned}
%  \label{dnu-dmu-pde}
  \end{equation*} 
Therefore, \eqref{log-pde-relative-density} is obtained after applying
  \eqref{lg-llogg} of Lemma~\ref{lemma-generator-property}.  
\end{proof}

Using Lemma~\ref{lemma-equation-of-log-density-nu-to-mu}, we can derive an
expression of the time derivative of $\mathcal{R}^{\mathrm{Bro}}(s)$,
i.e.\ the expression of the relative entropy production rate. 
\begin{prop}
  For all $s\in [0, T]$, we have 
\begin{align}
    \frac{d \mathcal{R}^{\mathrm{Bro}}(s)}{ds}  =& -\beta \int_{\mathbb{R}^n} \frac{\partial V}{\partial s} d\nu^{\infty}_s + \beta \int_{\mathbb{R}^n} \frac{\partial
    V}{\partial s}\, d\nu_s - 
  \frac{1}{\beta} \int_{\mathbb{R}^n}  \Big|\sigma^\top\nabla
  \left(\ln\frac{d\nu_s}{d\nu^{\infty}_s}\right)\Big|^2 \, d\nu_s\,. 
    \label{formula-dr}
\end{align}
  \label{prop-production-rate-of-entropy}
\end{prop}
\begin{proof}
  Using the Fokker-Planck equation \eqref{fokker-planck}, integration by parts
formula, the identities in \eqref{l-lr-dual}, as well as Lemma~\ref{lemma-equation-of-log-density-nu-to-mu}, 
   we can compute from \eqref{entropy-sr} that
\begin{align*}
    & \frac{d \mathcal{R}^{\mathrm{Bro}}(s)}{ds} \notag \\
    =& \int_{\mathbb{R}^n}
    \Big(\ln \frac{d\nu_s}{d\nu^{\infty}_s}\Big) \frac{\partial \rho}{\partial s}\, dx 
    +  \int_{\mathbb{R}^n} \rho\,\frac{\partial}{\partial
    s}\Big(\ln\frac{d\nu_s}{d\nu^{\infty}_s}\Big) \, dx \notag\\
    =& \int_{\mathbb{R}^n} \Big(\ln \frac{d\nu_s}{d\nu^{\infty}_s}\Big)
    \mathrm{e}^{-\beta V} \mathcal{L}_s^* (\mathrm{e}^{\beta V}\rho) \, dx \notag \\
    & +  \int_{\mathbb{R}^n}  \bigg[\beta \Big(\frac{\partial V}{\partial s}
  - \frac{d F}{ds} \Big) +
  \mathcal{L}_{s}^* \Big(\ln\frac{d\nu_s}{d\nu^{\infty}_s}\Big)+
    \frac{1}{\beta}\left|\sigma^\top\nabla
    \left(\ln\frac{d\nu_s}{d\nu^{\infty}_s}\right) \right|^2\bigg] \,\rho\, dx\notag \\
    =& - \beta \frac{d F(s)}{ds} + \int_{\mathbb{R}^n} \bigg[\big(\mathcal{L}_s  
    + \mathcal{L}_{s}^*\big) \Big(\ln\frac{d\nu_s}{d\nu^{\infty}_s}\Big)
  + \beta \frac{\partial V}{\partial s} 
    + \frac{1}{\beta} \left|\sigma^\top\nabla
    \left(\ln\frac{d\nu_s}{d\nu^{\infty}_s}\right)\right|^2\bigg] \,\rho \, dx\notag \\
    =& - \beta \frac{d F(s)}{ds} + \beta \int_{\mathbb{R}^n} \frac{\partial V}{\partial s} d\nu_s  
    + \int_{\mathbb{R}^n} 
    \bigg[\big(\mathcal{L}_s + \mathcal{L}_{s}^*\big)
    \Big(\ln\frac{d\nu_s}{d\nu^{\infty}_s}\Big) \bigg]
    \frac{d\nu_s}{d\nu^{\infty}_s} \,d\nu^{\infty}_s\notag \\
    & + \frac{1}{\beta} \int_{\mathbb{R}^n}\left|\sigma^\top\nabla
    \left(\ln\frac{d\nu_s}{d\nu^{\infty}_s}\right)\right|^2 \, d\nu_s \notag\\
    =& - \beta \frac{d F(s)}{ds} + \beta \int_{\mathbb{R}^n} \frac{\partial V}{\partial s} d\nu_s  
    - \frac{2}{\beta} \int_{\mathbb{R}^n} \gamma\nabla
    \Big(\frac{d\nu_s}{d\nu^{\infty}_s}\Big) \cdot \nabla
    \Big(\ln\frac{d\nu_s}{d\nu^{\infty}_s}\Big)\,
    \frac{d\nu^{\infty}_s}{d\nu_s}\, d\nu_s\notag \\
    & + \frac{1}{\beta} \int_{\mathbb{R}^n}\left|\sigma^\top\nabla
    \left(\ln\frac{d\nu_s}{d\nu^{\infty}_s}\right)\right|^2 \, d\nu_s \notag\\
    =& 
    -\beta \int_{\mathbb{R}^n} \frac{\partial V}{\partial s} d\nu^{\infty}_s 
    + \beta \int_{\mathbb{R}^n} \frac{\partial V}{\partial s} d\nu_s  
    - \frac{2}{\beta} \int_{\mathbb{R}^n} \gamma\nabla
    \Big(\ln\frac{d\nu_s}{d\nu^{\infty}_s}\Big) \cdot \nabla
    \Big(\ln\frac{d\nu_s}{d\nu^{\infty}_s}\Big)\,d\nu_s\notag \\
    & + \frac{1}{\beta} \int_{\mathbb{R}^n}\left|\sigma^\top\nabla
    \left(\ln\frac{d\nu_s}{d\nu^{\infty}_s}\right)\right|^2 \, d\nu_s \notag\\
    =& 
    -\beta \int_{\mathbb{R}^n} \frac{\partial V}{\partial s} d\nu^{\infty}_s + 
    \beta \int_{\mathbb{R}^n} \frac{\partial
    V}{\partial s}\, d\nu_s - 
\frac{1}{\beta} \int_{\mathbb{R}^n}  \left|\sigma^\top\nabla
    \left( \ln\frac{d\nu_s}{d\nu^{\infty}_s}\right)\right|^2 \, d\nu_s\,, 
\end{align*}
  where we used the fact that $\mathcal{L}_s^*$ 
  is the adjoint
  operator of $\mathcal{L}_s$ in $L^2(\mathbb{R}^n, \nu_{s}^{\infty})$ to get the third equality,  
 the expression $\frac{dF(s)}{ds} = \int_{\mathbb{R}^n}
\frac{\partial V}{\partial s} d\nu^{\infty}_s$, which can be verified using
\eqref{invariant-mu} and \eqref{free-energy-f}, and finally the relation
  $\gamma=\sigma\sigma^\top$ to derive the last equality.
\end{proof}
In particular, when the potential $V$ is time-independent, 
\eqref{formula-dr} becomes 
\begin{equation}
\frac{d \mathcal{R}^{\mathrm{Bro}}(s)}{ds} = - \frac{1}{\beta}
  \int_{\mathbb{R}^n}  \Big|\sigma^\top\nabla
    \ln\frac{d\nu_s}{d\nu^{\infty}_s}\Big|^2 \, d\nu_s \le 0\,, 
    \label{entropy-production-rate-v0}
\end{equation} 
which shows that the relative entropy $\mathcal{R}^{\mathrm{Bro}}(s)$ is non-increasing.
It is interesting to note that \eqref{entropy-production-rate-v0} is true for time-inhomogeneous processes $(x(s))_{s\in[0,T]}$ as long as $V$, or equivalently $\nu^{\infty}_s$, is time-independent (i.e.\ the vector $J$, coefficients $\gamma$ and $\sigma$ in \eqref{dynamics-1-q-vector} can be time-dependent).

We are ready to state Theorem~\ref{thm-entropy-estimate-brownian}.
\begin{thm}
  Assume that $\gamma$ satisfies the assumption \eqref{elliptic-gamma}
   and that the probability measure $\nu^{\infty}_s$ satisfies the logarithmic
   Sobolev inequality with (time-dependent, positive) constant
  $\kappa(s)$ at time $s \in [0,T]$. Let $L_1 ,L_2: [0,T]\rightarrow
  \mathbb{R}$ be two functions with nonnegative values.
  \begin{enumerate}
    \item
For all $s\in[0,T]$, suppose that $\frac{\partial V}{\partial s}(\cdot,s) \in
      \mathbf{B}^\infty(\mathbb{R}^n)$ with 
      $\|\frac{\partial V}{\partial s}(\cdot,s)\|_\infty \le L_1(s)$.  We have 
    \begin{align}
      \sqrt{\mathcal{R}^{\mathrm{Bro}}(s)} \le
      \mathrm{e}^{-\frac{\gamma^-}{\beta}\int_0^s \kappa(u)\,du} \sqrt{\mathcal{R}^{\mathrm{Bro}}(0)}
       + \frac{\beta }{\sqrt{2}} \int_0^s L_1(u)\,\mathrm{e}^{-\frac{\gamma^-}{\beta}\int_u^s
       \kappa(u')\,du'}\,du \,, \quad \forall~s \in [0,T]\,.
       \label{bound-on-r-1}
    \end{align}
    \item
      For all $s \in [0,T]$, suppose that $\frac{\partial V}{\partial s}(\cdot,s)$ is Lipschitz on
      $\mathbb{R}^n$ with Lipschitz constant $L_2(s)$.  We have, for all $s \in [0,T]$,
    \begin{align}
      \sqrt{\mathcal{R}^{\mathrm{Bro}}(s)} \le
      \mathrm{e}^{-\frac{\gamma^-}{\beta}\int_0^s \kappa(u)\,du} \sqrt{\mathcal{R}^{\mathrm{Bro}}(0)}
       + \frac{\beta}{\sqrt{2}} \int_0^s\frac{L_2(u)}{\sqrt{\kappa(u)}}\mathrm{e}^{-\frac{\gamma^-}{\beta}\int_u^s
       \kappa(u')\,du'}\,du \,.
       \label{bound-on-r-2}
    \end{align}
      \end{enumerate}
      \label{thm-entropy-estimate-brownian}
\end{thm}

\begin{proof}
  \begin{enumerate}
    \item
  Consider the terms in \eqref{formula-dr} of Proposition~\ref{prop-production-rate-of-entropy}.
      Since $\|\frac{\partial V}{\partial s}(\cdot , s)\|_\infty \le L_1(s)$, 
      applying \eqref{integration-bounded-by-tv} and Csisz\'ar-Kullback-Pinsker inequality \eqref{csiszar-kullback-pinsker}, we have 
  \begin{align}
    \begin{split}
      & \Big|\int_{\mathbb{R}^n} \frac{\partial V}{\partial s} d\nu^{\infty}_s - 
    \int_{\mathbb{R}^n} \frac{\partial
      V}{\partial s}\, d\nu_s\Big| \\
      \le & L_1(s)\,\|\nu^{\infty}_s-\nu_s\|_{TV} \\
      \le &
      L_1(s) \sqrt{2D_{KL}(\nu_s\,\|\,\nu^{\infty}_s)} \\
      =& L_1(s)\sqrt{2\mathcal{R}^{\mathrm{Bro}}(s)}\,.
    \end{split}
    \label{term1-bounded-by-sqrt-of-r}
  \end{align}
    Substituting \eqref{term1-bounded-by-sqrt-of-r} into \eqref{formula-dr},
      using the condition \eqref{elliptic-gamma}, the Fisher
      information~\eqref{fisher-info}, and the logarithmic Sobolev inequality \eqref{lsi},  
    we compute
    \begin{equation}
      \begin{aligned}
	\frac{d \mathcal{R}^{\mathrm{Bro}}(s)}{ds} 
	=\, & -\beta \int_{\mathbb{R}^n} \frac{\partial V}{\partial s} d\nu^{\infty}_s + \beta \int_{\mathbb{R}^n} \frac{\partial V}{\partial s}\, d\nu_s - 
\frac{1}{\beta} \int_{\mathbb{R}^n}  \left|\sigma^\top\nabla
	\left(\ln\frac{d\nu_s}{d\nu^{\infty}_s}\right)\right|^2 \, d\nu_s\\
	\le&\, \beta L_1(s) \sqrt{2\mathcal{R}^{\mathrm{Bro}}(s)} - 
	\frac{\gamma^-}{\beta} \mathcal{I}(\nu_s\,\|\,\nu^{\infty}_s)\\
	\le&\, \beta L_1(s) \sqrt{2\mathcal{R}^{\mathrm{Bro}}(s)} - \frac{2\kappa(s)\gamma^-}{\beta}
	\mathcal{R}^{\mathrm{Bro}}(s)\,,
    \end{aligned}
      \label{estimate1-in-prop1}
    \end{equation}
      which implies
      \begin{equation*}
	\frac{d}{ds} \sqrt{\mathcal{R}^{\mathrm{Bro}}(s)}
	\le - \frac{\kappa(s)\gamma^-}{\beta}
	\sqrt{\mathcal{R}^{\mathrm{Bro}}(s)} + \frac{\beta L_1(s)}{\sqrt{2}} \,.
      \end{equation*}
     The upper bound \eqref{bound-on-r-1} is implied by Gronwall's inequality. 
    \item
      Since $\frac{\partial V}{\partial s}$ is Lipschitz, 
      applying the inequalities \eqref{w1-bounded-by-w2} and \eqref{Lipschitz-w1}, we
      find
  \begin{align}
    & \Big|\int_{\mathbb{R}^n} \frac{\partial V}{\partial s} d\nu^{\infty}_s - 
    \int_{\mathbb{R}^n} \frac{\partial
    V}{\partial s}\, d\nu_s\Big| 
    \le  L_2(s)\,\mathbf{W}_1(\nu_s, \nu^{\infty}_s) 
    \le L_2(s) \mathbf{W}_2(\nu_s, \nu^{\infty}_s)\,.
    \label{estimate2-in-prop1}
  \end{align}
      Since $\nu^{\infty}_s$ satisfies the logarithmic Sobolev inequality with the constant $\kappa(s)$, the
      Talagrand inequality~\eqref{talagrand-ineq} holds, i.e.\
      \begin{align}
	\mathbf{W}_2(\nu_s, \nu^{\infty}_s) \le \sqrt{\frac{2}{\kappa(s)}
	D_{KL}(\nu_s\,\|\,\nu^{\infty}_s)}
	= \sqrt{\frac{2\mathcal{R}^{\mathrm{Bro}}(s)}{\kappa(s)}}\,.
    \label{estimate3-in-prop1}
      \end{align}
Applying~\eqref{estimate2-in-prop1}--\eqref{estimate3-in-prop1},
      a similar argument as \eqref{estimate1-in-prop1} shows that
    \begin{align*}
      \frac{d \mathcal{R}^{\mathrm{Bro}}(s)}{ds} \le \beta L_2(s)
      \sqrt{\frac{2\mathcal{R}^{\mathrm{Bro}}(s)}{\kappa(s)}} -
      \frac{2\kappa(s)\gamma^-}{\beta} \mathcal{R}^{\mathrm{Bro}}(s)\,.
    \end{align*}
     The upper bound \eqref{bound-on-r-2} again follows from Gronwall's inequality. 
  \end{enumerate}
\end{proof}

Note that the exponential decay in the time-homogeneous case is recovered
since $L_1(s)=L_2(s)=0$ and $\kappa(s)$ is time-independent.
In general, since the probability measure $\nu_s^\infty$ \eqref{invariant-mu}
  depends on $\beta$, the constant $\kappa(s)$ in
  Theorem~\ref{thm-entropy-estimate-brownian} depends on $\beta$ as well.
  In particular, when $\nabla^2 V(x, s) \ge \kappa_0(s) I_n$ for some
  $\kappa_0(s) > 0$ (Bakry-Emery criterion) which is independent of $\beta$, for all $(x,s) \in
  \mathbb{R}^n\times [0,T]$, then $\nu_s^\infty$ satisfies the logarithmic Sobolev inequality with the constant
  $\kappa(s)=\beta\kappa_0(s)$ (see Ref.~\onlinecite[Remark 21.4]{villani2008optimal}).
  The corresponding estimates in this case follow directly by substituting the expression $\kappa(s)=\beta\kappa_0(s)$ into \eqref{bound-on-r-1} and \eqref{bound-on-r-2}. 

We conclude this section with a remark on possible extensions of Theorem~\ref{thm-entropy-estimate-brownian}.
\begin{remark}
  Obviously, when the process is defined on $[0, +\infty)$, i.e.\ $T=+\infty$, the conclusions of
  Theorem~\ref{thm-entropy-estimate-brownian} are true on $[0, +\infty)$ as well.

In the proof above, the key step to derive upper bounds of $\mathcal{R}^{\mathrm{Bro}}(s)$ 
 is to estimate the difference $|\int_{\mathbb{R}^n} \frac{\partial V}{\partial s} d\nu^{\infty}_s - 
    \int_{\mathbb{R}^n} \frac{\partial V}{\partial s}\, d\nu_s|$. Instead of 
    assuming boundedness or Lipschitz condition for $\frac{\partial
    V}{\partial s}(\cdot, s)$, one can also assume that $\frac{\partial V}{\partial
    s}(\cdot, s)$ is
    $\alpha$-H{\"o}lder continuous for some $\alpha \in (0,1]$, i.e.\ 
  $\left|\frac{\partial V}{\partial s}(x,s) - \frac{\partial V}{\partial
  s}(y,s)\right| \le L(s)|x-y|^\alpha$ for all $x,y\in \mathbb{R}^n$ and $s
  \in [0,T]$, where $L(s) \ge 0$. 
    In this case, H\"{o}lder's inequality and Talagrand inequality (see \eqref{estimate3-in-prop1}) imply 
    \begin{equation*}
\left|\int_{\mathbb{R}^n} \frac{\partial V}{\partial s} d\nu^{\infty}_s - 
    \int_{\mathbb{R}^n} \frac{\partial V}{\partial s}\, d\nu_s\right| \le 
      L(s) \mathbf{W}_2(\nu_s, \nu_s^\infty)^{\alpha}
	\le L(s) \left(\frac{2\mathcal{R}^{\mathrm{Bro}}(s)}{\kappa(s)}\right)^{\frac{\alpha}{2}}\,,
    \end{equation*}
    from which we obtain
    \begin{align*}
      \frac{d \mathcal{R}^{\mathrm{Bro}}(s)}{ds} \le \beta L(s)
    \left(\frac{2\mathcal{R}^{\mathrm{Bro}}(s)}{\kappa(s)}\right)^{\frac{\alpha}{2}}
    - \frac{2\kappa(s)\gamma^-}{\beta} \mathcal{R}^{\mathrm{Bro}}(s)\,.  \end{align*}
  This again allows us to derive upper bounds of $\mathcal{R}^{\mathrm{Bro}}(s)$ using a Gronwall type inequality. We omit the details for simplicity.
  \label{rmk-holder-condition} 
\end{remark}
\subsection{Connection between time reversal of reverse process and optimally controlled forward process}
\label{subsec-crooks-control}
First of all, let us introduce a stochastic optimal control problem (see
Ref.~\onlinecite[Example III.8.2]{fleming2006}) that is related to Jarzynski's
equality. We refer to Appendix~\ref{app-subsec-control-jarzynski-brownian} for further motivations. Consider the stochastic optimal control problem 
\begin{align}
  U(x,t) = \inf_{(u_s)_{s\in [t,T]}} \mathbf{E}\left(W^u_{(t,T)} + \frac{1}{4} \int_t^T
  |u_s|^2\,ds\,\middle|\,x^u(t) = x\right)\,, \quad (x,t) \in \mathbb{R}^n\times [0, T]\,,
  \label{opt-control-problem-U}
\end{align}
of the controlled process 
\begin{align}
  \begin{split}
  d x^u(s)  =&  \Big(J - \gamma \nabla V + \frac{1}{\beta} \nabla \cdot \gamma\Big)(x^u(s), s)\,ds + \sqrt{2\beta^{-1}}
  \sigma(x^u(s),s)\,dw(s)\\
  & + \sigma(x^u(s), s)\, u_s\, ds \,,
  \end{split}
\label{dynamics-1-u}
\end{align}
 where the minimum is over all adapted processes $(u_s)_{s\in[t,T]}\in
 C([t,T],\mathbb{R}^m)$ ($u_s$ can depend on the past history $(x^u(r))_{0 \le
 r\le s}$ for all $s \in [t,T]$) such
 that \eqref{dynamics-1-u} has a strong solution, and the work functional is  
 \begin{align}
  W^u_{(t,T)} = \int_{t}^{T} \frac{\partial V}{\partial s}(x^u(s), s)\,ds\,,
   \quad ~t \in [0, T]\,.
  \label{work-w-u}
\end{align}

  The function $U: \mathbb{R}^n\times
 [0,T]\rightarrow \mathbb{R}$ is called the value function of the optimal
 control problem \eqref{opt-control-problem-U}--\eqref{dynamics-1-u}.
Under the condition~\eqref{elliptic-gamma} and mild technical conditions on coefficients, 
$U$ is the classical solution to the Hamilton-Jacobi-Bellman (HJB) equation~\eqref{hjb-eqn} in Appendix~\ref{app-subsec-control-jarzynski-brownian}. The optimal control (at which the infimum in \eqref{opt-control-problem-U} is achieved) exists and is given in feedback form as (see Ref.~\onlinecite[Section III.8]{fleming2006}) 
\begin{align}
   u_s^* = -2\sigma^\top(x,s) \nabla U(x,s)\,, \quad s \in [0,T]\,,
  \label{us-opt-change-of-measure}
\end{align}
when the system is at the state $x\in \mathbb{R}^n$ at time $s\in [0,T]$.
We also introduce the probability measure $\nu_0^*$ on $\mathbb{R}^n$, given by
\begin{align}
  \frac{d\nu_0^*}{dx}(x) 
  = \frac{\mathrm{e}^{-\beta V(x,0)}}{Z(T)} 
  \mathbf{E} \left( \mathrm{e}^{-\beta W}\,\middle|\,x(0)=x\right)\,, \quad x \in \mathbb{R}^n\,,
  \label{opt-mu0}
\end{align}
where (cf.\ \eqref{work-w-u})
  \begin{align}
  W = \int_{0}^{T} \frac{\partial V}{\partial s}(x(s), s)\,ds\,,
  \label{work-w}
\end{align}
$Z(T)$ is given in \eqref{invariant-mu},
and $\mathbf{E}(\cdot\,|\,x(0)=x)$ denotes the
path ensemble average of the (uncontrolled) dynamics
\eqref{dynamics-1-q-vector} starting from $x$ at $s=0$.
Note that Jarzynski's equality \eqref{jarzynski} in Appendix~\ref{app-subsec-control-jarzynski-brownian} implies that the integration of
$\nu_0^*$ over $\mathbb{R}^n$ equals to one. Moreover, the optimal
control~$u^*$ \eqref{us-opt-change-of-measure}
and the probability measure $\nu_0^*$ \eqref{opt-mu0} give the optimal Monte
Carlo estimators for free energy calculations based on Jarzynski's equality (see
Appendix~\ref{app-subsec-control-jarzynski-brownian} and Ref.~\onlinecite{non-equilibrium-2018}). 

The following result connects the time reversal of the reverse process \eqref{dynamics-1-q-vector-backward}
  and the controlled process~\eqref{dynamics-1-u} under the optimal control~\eqref{us-opt-change-of-measure}. 
\begin{thm}
  Let $\nu^{\infty}_T$ and $\nu^{*}_0$ be the probability distributions in
  \eqref{invariant-mu} and \eqref{opt-mu0}, respectively.
  Consider the reverse process $(x^R(s))_{s\in[0,T]}$ \eqref{dynamics-1-q-vector-backward}
  starting from the distribution $\nu^{\infty}_T$ at $s=0$. 
  Assume the probability distribution of $x^R(s)\in \mathbb{R}^n$ has a positive 
  $C^\infty$-smooth density with respect to Lebesgue measure for all $s \in [0,T]$.
  Define the time reversal $(x^{R,-}(s))_{s\in[0,T]}$, where $x^{R,-}(s)=x^R(T-s)$ for $s\in
  [0,T]$. Let $(x^{u^*}(s))_{s\in[0,T]}$ be the controlled process \eqref{dynamics-1-u} under the optimal control $u^*$ \eqref{us-opt-change-of-measure} starting from the initial distribution $\nu_0^*$.
  Then, the processes $(x^{R,-}(s))_{s\in[0,T]}$ and $(x^{u^*}(s))_{s\in[0,T]}$ have the same law on the path space $C([0,T], \mathbb{R}^n)$.
  \label{thm-connection-brownian}
\end{thm}
\begin{proof}
  Let us denote by $\nu_s^R$ the probability distribution of $x^R(s)\in \mathbb{R}^n$ at time
  $s$. Let $\rho^R(\cdot,s)$ be the probability density of $\nu_s^R$ with
  respect to Lebesgue measure. Similar to \eqref{fokker-planck}, the density $\rho^R$ satisfies the Fokker-Planck equation
 \begin{align}
   \frac{\partial \rho^R}{\partial s} = \mathrm{e}^{-\beta V(\cdot, T-s)}
   (\mathcal{L}^R_s)^* (\mathrm{e}^{\beta V(\cdot, T-s)}\rho^R)\,, \quad
   s\in [0,T]\,,
  \label{fokker-planck-r}
\end{align}
  where $(\mathcal{L}^R_s)^*$ denotes the adjoint operator of
  $\mathcal{L}^R_s$ \eqref{l-r-s} with respect to the probability measure~$\nu_{T-s}^\infty$. 

  In the following, we show that both $(x^{R,-}(s))_{s\in[0,T]}$ and
  $(x^{u^*}(s))_{s\in[0,T]}$ satisfy the same SDE with the same initial probability distribution. 
  \begin{enumerate}
    \item
      First, we consider the SDEs of $(x^{R,-}(s))_{s\in[0,T]}$ and
      $(x^{u^*}(s))_{s\in[0,T]}$.
      For the optimally controlled process $(x^{u^*}(s))_{s\in[0,T]}$, combining
      \eqref{dynamics-1-u} and \eqref{us-opt-change-of-measure}, using
      $\gamma=\sigma\sigma^\top$, we find
\begin{align}
  \begin{split}
    d x^{u^*}(s)  =&  \Big(J - \gamma \nabla V + \frac{1}{\beta} \nabla \cdot
    \gamma\Big)(x^{u^*}(s), s)\,ds + \sqrt{2\beta^{-1}} \sigma(x^{u^*}(s),s)\,dw(s)\\
    & - 2(\gamma\, \nabla U)(x^{u^*}(s), s)\, ds \,, \quad s \in [0,T]\,,
  \end{split}
\label{dynamics-1-optimal}
\end{align}
  where $U$ is the value function in \eqref{opt-control-problem-U}. 

      For the time reversal $(x^{R,-}(s))_{s\in[0,T]}$, let us recall the reverse process \eqref{dynamics-1-q-vector-backward}, which we rewrite as 
\begin{align}
    d x^R(s) =  \widehat{b}(x^R(s),s)\,ds + \sqrt{2\beta^{-1}}
    \widehat{\sigma}(x^R(s),s)\,dw(s)\,,
    \label{reverse-xs-rewrite}
\end{align}
  where, for all $(x,s) \in \mathbb{R}^n \times [0,T]$, 
  \begin{equation}
    \begin{aligned}
      & \widehat{b}(x,s) = \Big(-J - \gamma \nabla V + \frac{1}{\beta} \nabla \cdot
    \gamma\Big)(x,T-s)\,, \\
      & \widehat{\sigma}(x,s) = \sigma(x,T-s)\,,\quad \widehat{\gamma}(x,s) =
      (\widehat{\sigma}\widehat{\sigma}^\top)(x,s) = \gamma(x,T-s)\,.
      \end{aligned}
    \label{b-sigma-rewrite}
  \end{equation}

  Since the density $\rho^R$ is both $C^\infty$-smooth and positive, it is
      shown in Ref.~\onlinecite{haussmann1986} that
      $(x^{R,-}(s))_{s\in[0,T]}=(x^R(T-s))_{s\in[0,T]}$ is again a diffusion
      process and satisfies the SDE 
\begin{align}
  dx^{R,-}(s) = \widehat{b}^{\,-}(x^{R,-}(s),T-s)\,ds + \sqrt{2\beta^{-1}}
  \widehat{\sigma}(x^{R,-}(s),T-s)\,dw(s)\,, \quad s\in [0,T]\,,
    \label{reversal-of-reverse-1}
\end{align}
      with the initial distribution $\nu^R_T$ (i.e.\ the distribution of
      $x^R(T)\in \mathbb{R}^n$), where 
the drift term $\widehat{b}^{\,-}$ is 
\begin{equation}
    \widehat{b}^{\,-}(x, s) = \Big(-\widehat{b} + \frac{2}{\beta \rho^R}
    \nabla\cdot(\rho^R~\widehat{\gamma})\Big)(x,s)\,,  \quad ~(x,s) \in \mathbb{R}^n\times
    [0,T]\,.
    \label{b-in-reversal}
\end{equation}
  Substituting \eqref{b-sigma-rewrite} in \eqref{b-in-reversal}, we can derive
  \begin{equation}
    \begin{aligned}
      \widehat{b}^{\,-}(x, T-s) =& \Big(-\widehat{b} + \frac{2}{\beta \rho^R}
    \nabla\cdot(\rho^R~\widehat{\gamma})\Big)(x,T-s)\\
      = & \Big(J + \gamma \nabla V - \frac{1}{\beta} \nabla \cdot
      \gamma + \frac{2}{\beta \rho^{R,-}} \nabla\cdot(\rho^{R,-}\gamma) \Big)(x,s) \\
      = & \Big(J + \gamma \nabla V 
      - \frac{1}{\beta} \nabla \cdot \gamma 
      + \frac{2}{\beta} \nabla \cdot \gamma + \frac{2}{\beta \rho^{R,-}}
      \gamma \nabla\rho^{R,-} \Big)(x,s) \\
      = & \Big(J + \gamma \nabla V + \frac{1}{\beta} \nabla \cdot
      \gamma + \frac{2}{\beta}\gamma \nabla \ln\rho^{R,-} \Big)(x,s) \\
      = & \Big(J - \gamma \nabla V + \frac{1}{\beta} \nabla \cdot
      \gamma + \frac{2}{\beta} \gamma\nabla \ln(\mathrm{e}^{\beta V}\rho^{R,-}) \Big)(x,s)\,,
    \end{aligned}
    \label{b-in-reversal-explicit}
  \end{equation}
  where we have used the notation $\rho^{R,-}(\cdot,s) = \rho^R(\cdot,T-s)$,
  for all $s \in [0,T]$.
Using \eqref{b-in-reversal-explicit} and \eqref{b-sigma-rewrite}, 
we can write the SDE \eqref{reversal-of-reverse-1} more explicitly as 
\begin{align}
  \begin{split}
    dx^{R,-}(s) =& \widehat{b}^{\,-}(x^{R,-}(s),T-s)\,ds + \sqrt{2\beta^{-1}}
  \widehat{\sigma}(x^{R,-}(s),T-s)\,dw(s) \\
    =& \Big(J - \gamma \nabla V + \frac{1}{\beta} \nabla \cdot
      \gamma \Big)(x^{R,-}(s) ,s)\,ds +
    \sqrt{2\beta^{-1}} \sigma\big(x^{R,-}(s),s\big)\,dw(s) \\
    & + \frac{2}{\beta} \Big(\gamma\nabla \ln(\mathrm{e}^{\beta V}\rho^{R,-})\Big) (x^{R,-}(s) ,s)\,ds\,.
  \end{split}
  \label{reversal-of-reverse-2}
\end{align}

To show that the two SDEs \eqref{dynamics-1-optimal} and \eqref{reversal-of-reverse-2} are the same, 
we define 
\begin{equation}
g(x,s) =\mathrm{e}^{\beta V(x,s)} \rho^R(x, T-s)Z(T),\, \quad (x,s) \in
  \mathbb{R}^n\times [0,T]\,.
  \label{def-g}
\end{equation}
  Since the initial distribution of $x^R(s)\in \mathbb{R}^n$ is $\nu_0^R=\nu^{\infty}_T$, we have $\rho^R(x,0) =
  \frac{1}{Z(T)}\mathrm{e}^{-\beta V(x,T)}$, which implies $g(x,T) =1$. Using
  \eqref{fokker-planck-r} and the identity
  $\mathcal{L}_s^*=\mathcal{L}_{T-s}^R$ in
  Lemma~\ref{lemma-generator-property}, we can derive
  \begin{equation}
    \begin{aligned}
      \frac{\partial g}{\partial s} =& - Z(T) \mathrm{e}^{\beta V(\cdot,s)} \frac{\partial \rho^R}{\partial
    s}(\cdot,T-s) + \beta \frac{\partial
      V}{\partial s} g \\
      =& -Z(T) (\mathcal{L}^R_{T-s})^* \big(\mathrm{e}^{\beta V(\cdot, s)}\rho^R(\cdot, T-s)\big) + \beta \frac{\partial
      V}{\partial s} g \\
      =& -\mathcal{L}_s g  + \beta \frac{\partial
      V}{\partial s} g \,.
    \end{aligned}
    \label{pde-for-g}
  \end{equation}
  In fact, the derivations above show that $g$ and the value function $U$ are
  related by a logarithmic transformation
  (we refer to \eqref{g-pde}--\eqref{hjb-eqn} in
  Appendix~\ref{app-subsec-control-jarzynski-brownian} for details), i.e.\ 
  \begin{equation}
  U=-\beta^{-1} \ln g\,.
    \label{log-transformation-u-and-g}
  \end{equation}
  
Combining \eqref{def-g} and \eqref{log-transformation-u-and-g}, we see that
both SDEs \eqref{dynamics-1-optimal} and \eqref{reversal-of-reverse-2} are the same.
    \item
      Next, we show that the initial distribution of $x^{R,-}(s)\in
      \mathbb{R}^n$ is $\nu^{*}_0$ in \eqref{opt-mu0}. 
  Since $(x^{R,-}(s))_{s\in[0,T]}=(x^R(T-s))_{s\in[0,T]}$, it is enough to
  verify $\nu_T^R=\nu^{*}_0$ or, equivalently, the density $\rho^R(x, T)$ coincides with the one in \eqref{opt-mu0}.
  In fact, applying Feynman-Kac formula, from \eqref{pde-for-g} we obtain 
\begin{align}
  g(x,s) = \mathbf{E}\Big[\exp\Big(-\beta\int_s^T \frac{\partial V}{\partial s}(x(t), t) dt\Big)\, \Big|\, x(s) = x\Big]\,,
  \quad (x,s) \in \mathbb{R}^n \times [0,T]\,,
  \label{g-x-t}
\end{align}
where $\mathbf{E}(\cdot\,|\,x(s)=x)$ denotes the path ensemble average of the (uncontrolled) dynamics
\eqref{dynamics-1-q-vector} starting from $x$ at time $s$.
Using \eqref{def-g} and \eqref{g-x-t}, we obtain 
\begin{equation}
  \rho^R(x,s)= \frac{\mathrm{e}^{-\beta V(x,T-s)}}{Z(T)} 
  \mathbf{E}\Big[\exp\Big(-\beta\int_{T-s}^T \frac{\partial V}{\partial
  s}(x(t), t) dt\Big)\, \Big|\, x(T-s) = x\Big]\,,
    \label{solution-of-rho-reverse}
\end{equation}
which implies (using \eqref{opt-mu0} and \eqref{work-w}) that 
\begin{equation}
d\nu_T^R = \rho^R(x,T)\,dx=\frac{1}{Z(T)}\,\mathrm{e}^{-\beta V(x,0)} g(x,0)\,dx
=d\nu^{*}_0. 
    \label{calculation-thm-connection-brownian-1}
\end{equation}
  This shows that the initial distribution of $(x^{R,-}(s))_{s\in[0,T]}$ is indeed $\nu^{*}_0$.  
\end{enumerate}
  To summarize, both $(x^{R,-}(s))_{s\in[0,T]}$ and $(x^{u^*}(s))_{s\in[0,T]}$ satisfy the same SDE with the same
  initial distribution $\nu^{*}_0$. Therefore, they have the same law on the path space.
\end{proof}

\section{Time-inhomogeneous Langevin dynamics}
\label{sec-langevin}
In this section, we study time-inhomogeneous (underdamped) Langevin dynamics in phase space. 
The analysis is similar to that in Section~\ref{sec-overdamped} for Brownian dynamics.
First, we introduce Langevin dynamics and useful notation in Section~\ref{subsec-forward-backward-langevin}.
Then, in Section~\ref{subsec-entropy-langevin} we study the relative entropy estimate
for time-inhomogeneous Langevin dynamics. Finally,  
in Section~\ref{subsec-crooks-control-langevin} we establish the connection between the time
reversal of reverse Langevin process and certain optimally controlled forward Langevin process.

\subsection{Forward and reverse processes}
\label{subsec-forward-backward-langevin}

First, we discuss the forward Langevin process which will be studied in Section~\ref{subsec-entropy-langevin}. Denote by $(q,p)$ the state of the system in phase space $\mathbb{R}^n \times \mathbb{R}^n$.
Let $\nabla_q$, $\nabla_p$ be the gradient operators with respect to 
positions and momenta components, respectively. 
Given a time-dependent $C^\infty$-smooth Hamiltonian $H: \mathbb{R}^{n} \times \mathbb{R}^{n} \times
[0, T] \rightarrow \mathbb{R}$ within time $[0,T]$, we consider the forward
Langevin dynamics (see Ref.~\onlinecite[Section 2.2.3]{tony-free-energy-compuation})
\begin{align}
  \begin{split}
    dq(s) =& \nabla_p H(q(s),p(s), s)\,ds \\
    dp(s) =& -\nabla_q H(q(s),p(s), s)\,ds - \gamma(q(s),s)
    \nabla_pH(q(s),p(s), s)\,ds + \sqrt{2\beta^{-1}} \sigma(q(s),s)\,dw(s)
  \end{split}
  \label{langevin-eqn-forward}
\end{align}
for $s\in [0,T]$, where $\sigma: \mathbb{R}^n \times [0,T]\rightarrow
\mathbb{R}^{n\times m}$ is $C^\infty$-smooth, $(w(s))_{s\in[0,T]}$ is an $m$-dimensional Brownian motion,
and $\beta>0$ is related to the (inverse) temperature of the system. We assume
that $\gamma=\sigma\sigma^\top: \mathbb{R}^n\times [0,T]\rightarrow
\mathbb{R}^{n\times n}$ and that the condition \eqref{elliptic-gamma} holds.
Note that both $\sigma$ and $\gamma$ are independent of momenta $p$.

The generator of \eqref{langevin-eqn-forward} at \textit{fixed} time $s \in [0,T]$ is given by 
\begin{align}
  \begin{split}
    \mathcal{Q}_s f =& \nabla_pH\cdot \nabla_q f - \nabla_q H \cdot \nabla_p f -
  \gamma\nabla_pH\cdot\nabla_p f + \frac{1}{\beta} \gamma:\nabla^2_p f\\
  =& \nabla_pH\cdot \nabla_q f - \nabla_q H \cdot \nabla_p f +
  \frac{\mathrm{e}^{\beta H}}{\beta} \mbox{div}_p\Big(\mathrm{e}^{-\beta H}
    \gamma \nabla_p f\Big)\,,
  \end{split}
  \label{l-s-generator-langevin-forward}
\end{align}
for a test function $f: \mathbb{R}^n \times \mathbb{R}^n \rightarrow \mathbb{R}$.
We assume that the time-homogeneous Langevin dynamics (similar to \eqref{forward-dynamics-intro-fixed-s} in the case of Brownian dynamics) with Hamiltonian $H(\cdot, \cdot, s)$, whose generator is $\mathcal{Q}_s$, is ergodic with respect to the unique
invariant measure $\pi^\infty_s$ on $\mathbb{R}^n \times \mathbb{R}^n$, given by 
\begin{align}
  d\pi^\infty_s = \frac{1}{\mathcal{Z}(s)} \mathrm{e}^{-\beta H(q,p,s)}\,dqdp\,,\quad
  \mbox{where}~\mathcal{Z}(s) =
  \int_{\mathbb{R}^n\times \mathbb{R}^n} \mathrm{e}^{-\beta H(q,p,s)}\,dqdp\,.
  \label{mu-s-langevin}
\end{align}
We refer to Refs.~\onlinecite{Mattingly2002,sachs-Leimkuhler-entropy2017} as well as 
Ref.~\onlinecite[Section 2.2.3]{tony-free-energy-compuation} for
sufficient conditions when $H$ is in the standard form
\eqref{standard-Hamiltonian} below, and to Ref.~\onlinecite{langevin-with-general-kinetic-energy} for conditions with a general separable $H$.
Let $\mathcal{Q}^*_s$ be the adjoint operator of $\mathcal{Q}_s$ in
$L^2(\mathbb{R}^n\times \mathbb{R}^n, \pi_{s}^{\infty})$ 
(i.e.\ with respect to the weighted inner product defined by $\pi_{s}^{\infty}$
\eqref{mu-s-langevin}).
Integrating by parts and using \eqref{l-s-generator-langevin-forward}, we find
\begin{equation}
  \begin{aligned}
  \mathcal{Q}^*_s f = & -\nabla_pH\cdot \nabla_q f + \nabla_q H \cdot \nabla_p f +
  \frac{\mathrm{e}^{\beta H}}{\beta} \mbox{div}_p\Big(\mathrm{e}^{-\beta H}
    \gamma \nabla_p f\Big)\,,
  \end{aligned}
  \label{q-adjoint-langevin}
\end{equation}
which differs from $\mathcal{Q}_s$ \eqref{l-s-generator-langevin-forward} by a
change of sign in the first two terms (i.e.\ a change of sign in the generator of the Hamiltonian part of
\eqref{langevin-eqn-forward}). 
Denote by $\pi_s$ the probability measure of the state $(q(s),p(s))\in \mathbb{R}^n\times \mathbb{R}^n$ 
of \eqref{langevin-eqn-forward} at time $s\in[0,T]$.
We assume that $\pi_s$ has a positive $C^\infty$-smooth probability density
$\varrho$ with respect to Lebesgue measure (see
Remark~\ref{rmk-positivity-langevin-gaussian-case} and
Example~\ref{example-gaussian-langevin} below), such that, for $s \in [0,T]$,
\begin{align}
  d\pi_s = \varrho(q,p,s)\,dqdp, \hspace{1cm} \mbox{where}~~ \int_{\mathbb{R}^n\times
  \mathbb{R}^n} \varrho(q,p,s)\,dqdp=1\,.
  \label{pi-rho-density}
\end{align}
 In this case, $\varrho$ satisfies the Fokker-Planck equation (cf. \eqref{fokker-planck})
\begin{align}
  \frac{\partial \varrho}{\partial s} = \mathrm{e}^{\beta H}\mathcal{Q}^*_s
  \big(\mathrm{e}^{-\beta H} \varrho\big) \,,\quad s \in [0,T]
  \label{fokker-planck-langevin}
\end{align}
in a classical sense. The free energy of the system is 
\begin{align}
  \mathcal{F}(s) = -\beta^{-1}\ln \mathcal{Z}(s)\,,\quad  s\in[0,T]\,.
  \label{free-energy-f-langevin}
\end{align}

Note that \eqref{langevin-eqn-forward} recovers the standard time-homogeneous Langevin dynamics
\begin{align}
  \begin{split}
    dq(s) =& M^{-1}p(s)\,ds \\
    dp(s) =& -\nabla_q V(q(s))\,ds - \gamma(q(s))
    M^{-1}p(s)\,ds + \sqrt{2\beta^{-1}} \sigma(q(s))\,dw(s)\,,
  \end{split}
  \label{langevin-eqn-forward-case-1}
\end{align}
when both $\sigma$ and $\gamma$ are independent of $s$, and the Hamiltonian is 
    \begin{align}
      H(q,p) = V(q) + \frac{p^\top M^{-1}p}{2}, \quad ~(q,p) \in \mathbb{R}^n\times \mathbb{R}^n\,,
      \label{standard-Hamiltonian}
    \end{align} 
      where $V:\mathbb{R}^n \rightarrow
    \mathbb{R}$ is a $C^\infty$-smooth potential function, and $M\in
    \mathbb{R}^{n\times n}$ is a constant symmetric positive definite
    matrix (see Refs.~\onlinecite{kinetic-choice-of-hmc,langevin-with-general-kinetic-energy}
    for more general settings with non-quadratic kinetic energies).

Next, we introduce the reverse Langevin dynamics which will be the subject of
Section~\ref{subsec-crooks-control-langevin}. Corresponding to
\eqref{langevin-eqn-forward}, the reverse Langevin dynamics is defined as
\begin{align}
  \begin{split}
    dq^R(s) =& -\nabla_p H(q^R(s),p^R(s), T-s)\,ds \\
    dp^R(s) =& \nabla_q H(q^R(s), p^R(s), T-s)\,ds - \gamma(q^R(s),T-s) \nabla_pH(q^R(s),
    p^R(s), T-s)\,ds \\
    & + \sqrt{2\beta^{-1}} \sigma(q^R(s), T-s)\,dw(s)
  \end{split}
  \label{langevin-eqn-backward}
\end{align}
for $s \in [0, T]$, where $(w(s))_{s\in[0,T]}$ denotes possibly another (independent) $m$-dimensional Brownian motion.
Note that, comparing to \eqref{langevin-eqn-forward}, there is a change of sign in the Hamiltonian part of 
\eqref{langevin-eqn-backward}. For \textit{fixed} $s \in [0,T]$, the generator of \eqref{langevin-eqn-backward} at time $s$ is 
\begin{align}
  \begin{split}
    \mathcal{Q}^{R}_s f=& -\nabla_pH(q,p,T-s)\cdot \nabla_q f + \nabla_q H(q,p,T-s)
  \cdot \nabla_p f\\
  &-
    \big[\gamma(q,T-s)\nabla_pH(q,p,T-s)\big]\cdot\nabla_p f + \frac{1}{\beta}
    \gamma(q,T-s):\nabla^2_p f\,,
  \end{split}
  \label{l-s-generator-langevin-backward}
\end{align}
for a test function $f: \mathbb{R}^n \times \mathbb{R}^n \rightarrow \mathbb{R}$.
The following results can be directly verified. We omit its proof since it is
similar to the proof of Lemma~\ref{lemma-generator-property}.
\begin{lemma}
  Given $f,g\in C^2(\mathbb{R}^n \times \mathbb{R}^n)$, for all $s \in
  [0,T]$, we have $\mathcal{Q}^*_s  = \mathcal{Q}^R_{T-s}$ and 
\begin{align}
     \int_{\mathbb{R}^n\times \mathbb{R}^n}
    f\big(\mathcal{Q}_s + \mathcal{Q}^*_s\big)g \,d\pi^\infty_s =
    -\frac{2}{\beta} \int_{\mathbb{R}^n\times \mathbb{R}^n} \gamma\nabla_p
    f\cdot\nabla_p\,g\, d\pi^\infty_s\,.
  \label{l-lr-dual-langevin}
\end{align}
Furthermore, when $g$ is positive, we have 
\begin{align}
  \begin{split}
    \frac{\mathcal{Q}_s g}{g}  =&\, \mathcal{Q}_s (\ln g) +\frac{1}{\beta}
    |\sigma^\top\nabla_p\ln g|^2\,, \\
    \frac{\mathcal{Q}^*_{s} g}{g}  =&\, \mathcal{Q}^*_{s} (\ln g)
    +\frac{1}{\beta} |\sigma^\top\nabla_p\ln g|^2 \,.
  \end{split}
  \label{lg-llogg-langevin}
\end{align}
  \label{lemma-generator-property-langevin}
\end{lemma}
Note that \eqref{lg-llogg-langevin} has been used in
Ref.~\onlinecite{Iacobucci2019} to study Langevin dynamics under external
forcing.  Similar to the Brownian dynamics case, Lemma~\ref{lemma-generator-property-langevin} can be used to give a simple derivation of the fluctuation relation between
the forward Langevin dynamics \eqref{langevin-eqn-forward} and the reverse
Langevin dynamics \eqref{langevin-eqn-backward}. We refer to Ref.~\onlinecite{non-equilibrium-2018} for details. 

Before concluding, we discuss the smoothness and positivity of
the density $\varrho$ \eqref{pi-rho-density}.
\begin{remark}
  Due to the degeneracy of the noise term in \eqref{langevin-eqn-forward}, the smoothness and the positivity
  of the density $\varrho$ \eqref{pi-rho-density} are less apparent comparing
  to the Brownian dynamics case in Section~\ref{sec-overdamped}. In the
  time-inhomogeneous setting, the smoothness of densities has been studied in Refs.~\onlinecite{mufa-chen-time-dependent,Hypoelliptic-non-homogeneous-2002}
using Malliavin calculus. The positivity of $\varrho$ follows by
  representing it as a path ensemble average of the reverse process \eqref{langevin-eqn-backward} using 
  Feynman-Kac formula.
  See the representation \eqref{varrho-represented-by-feynmann-kac-langevin}
  in Section~\ref{subsec-crooks-control-langevin} for the density of \eqref{langevin-eqn-backward}
  starting from $\pi^\infty_T$.
  We also refer to the proof of Ref.~\onlinecite[Theorem
  3.3.6.1]{michel-pardoux-1990}, which can be extended to the time-inhomogeneous setting. 
  \label{rmk-positivity-langevin-gaussian-case}
\end{remark}

We give an example where the positivity of the densities
can be seen directly. 
\begin{example}[Langevin dynamics under time-dependent Gaussian potential]
Consider the Hamiltonian $H(q,p,s) = \frac{1}{2}\eta(s)|q|^2 + \frac{p^\top M^{-1}p}{2}$ for some 
      continuous function $\eta: [0,T]\rightarrow \mathbb{R}^+$ taking
      positive values, and $\sigma=\gamma=I_n$. 
In this case, \eqref{langevin-eqn-forward} is a linear SDE with time-dependent
  coefficients, whose solution can be expressed as  
      \begin{equation}
  \begin{pmatrix}
    q(s)\\
    p(s)\\
  \end{pmatrix}
	= \Gamma(s)^{-1} \begin{pmatrix}
    q(0)\\
    p(0)\\
  \end{pmatrix}
	+ 
	\Gamma(s)^{-1} \bigintsss_0^s \Gamma(t) 
	\begin{pmatrix}
	   0 \\
	   dw(t)
	\end{pmatrix}
	  \,,\quad s \in [0,T]\,,
	\label{qp-sol-gaussian}
      \end{equation}
      where $\Gamma:[0,T]\rightarrow \mathbb{R}^{2n \times 2n}$  solves the
      matrix-valued ordinary differential equation (ODE)
      \begin{equation}
	  \frac{d\Gamma(s)}{ds} = -\Gamma(s)\Sigma(s)\,, ~ \forall s \in
	[0,T]\,, \quad 
	\mbox{for} ~ \Sigma(s) = \begin{pmatrix}
	  0 & -M^{-1} \\
	  \eta(s)I_n & M^{-1} 
	\end{pmatrix} \in \mathbb{R}^{2n \times 2n}\,,
	\label{matrix-ode-gaussian}
      \end{equation}
      with initial condition $\Gamma(0) = I_{2n}$. Applying Jacobi's
      formula (for derivatives of matrix determinants) to \eqref{matrix-ode-gaussian} gives
      $\frac{d\det\Gamma(s)}{ds} = -\Tr(\Sigma(s)) \det\Gamma(s) = -\Tr(M^{-1}) \det\Gamma(s)$, from which we
      get $\det\Gamma(s) = \mathrm{e}^{-\Tr(M^{-1}) s} > 0$ and therefore
      $\Gamma(s)$ is invertible for $s \in [0,T]$. Since the state $(q(s),p(s))$ \eqref{qp-sol-gaussian}
      is the sum of the image of $(q(0),p(0))$ under the invertible
      (one-to-one) linear map $\Gamma(s)$ and a centered Gaussian random variable, it is straightforward to 
      see that the density of $(q(s),p(s))$ at time $s\in [0,T]$ is positive everywhere
      whenever this is true for the initial density at time $s=0$.
      \label{example-gaussian-langevin}
\end{example}

\subsection{Relative entropy estimate}
\label{subsec-entropy-langevin}
Given $s \in [0, T]$, recall that $\pi^\infty_s$ and $\pi_s$ are the probability measures defined in
\eqref{mu-s-langevin} and \eqref{pi-rho-density}, respectively.
The goal of this section is to estimate the relative entropy 
\begin{align}
  \mathcal{R}^{\mathrm{Lan}}(s) = D_{KL}(\pi_s\,\|\,\pi^\infty_s) 
  = \int_{\mathbb{R}^n\times\mathbb{R}^n} \varrho(q,p,s)\, \ln
  \frac{d\pi_s}{d\pi^{\infty}_s}(q,p)\,  dqdp \,.
  \label{entropy-sr-langevin}
\end{align}
For simplicity we consider the case where 
  \begin{equation}
    \sigma=\sqrt{\xi} I_n\,,\quad 
  H(q,p,s) = V(q,s) + \frac{|p|^2}{2}, \quad (q,p,s) \in \mathbb{R}^n\times
  \mathbb{R}^n \times [0, T]\,,
    \label{special-hamiltonian}
\end{equation}
for some constant $\xi>0$ and some time-dependent $C^2$-smooth potential function 
$V: \mathbb{R}^n\times [0,T]\rightarrow \mathbb{R}$. Then,
\eqref{mu-s-langevin} becomes $d\pi^\infty_s=\mathcal{Z}(s)^{-1}
\mathrm{e}^{-\beta (V(q,s)+ |p|^2/2)}dqdp$, whose marginal
probability measure in position $q$ is (cf.\ \eqref{invariant-mu})
\begin{align}
  d\nu^\infty_s(q) = \frac{1}{Z(s)} \mathrm{e}^{-\beta V(q,s)}\,dq\,, \quad
  \mbox{where}~~ Z(s) = \int_{\mathbb{R}^n} \mathrm{e}^{-\beta V(q,s)}\,dq \,.
  \label{mu-s-langevin-q-margin}
\end{align}

Let us first state the following two results, which in fact hold for Langevin dynamics in the general form \eqref{langevin-eqn-forward}. 
We omit their proofs since they are similar to the proofs of
Lemma~\ref{lemma-equation-of-log-density-nu-to-mu} and
Proposition~\ref{prop-production-rate-of-entropy} in
Section~\ref{subsec-entropy-brownian}, respectively.
\begin{lemma}
  For $s \in [0,T]$, we have 
\begin{align}
  \frac{\partial }{\partial s} \left(\ln\frac{d\pi_s}{d\pi^\infty_s}\right)
  =& \beta \Big(\frac{\partial H}{\partial s} - \frac{d \mathcal{F}}{ds} \Big)
   + \mathcal{Q}_{s}^* \Big(\ln\frac{d\pi_s}{d\pi^\infty_s}\Big)+
  \frac{1}{\beta}\left|\sigma^\top\nabla_p \Big(\ln\frac{d\pi_s}{d\pi^\infty_s}\Big)\right|^2 \,,
  \label{log-pde-relative-density-langevin}
\end{align}
  where $\mathcal{F}$ is the free energy \eqref{free-energy-f-langevin},
  and $\mathcal{Q}_{s}^*$ is the adjoint operator \eqref{q-adjoint-langevin}.
  \label{lemma-equation-of-log-density-nu-to-mu-langevin}
\end{lemma}
\begin{prop}
  For $s \in [0,T]$, we have
\begin{align}
  \begin{split}
    \frac{d \mathcal{R}^{\mathrm{Lan}}(s)}{ds}  = & -\beta
    \int_{\mathbb{R}^n\times \mathbb{R}^n} \frac{\partial H}{\partial s}\,
    d\pi^\infty_s + \beta \int_{\mathbb{R}^n\times \mathbb{R}^n} \frac{\partial
    H}{\partial s}\, d\pi_s - 
    \frac{1}{\beta} \int_{\mathbb{R}^n\times \mathbb{R}^n}
    \left|\sigma^\top\nabla_p \Big(\ln\frac{d\pi_s}{d\pi^\infty_s}\Big)\right|^2 \, d\pi_s\,. 
  \end{split}
  \label{production-rate-of-entropy-langevin}
\end{align}
  \label{prop-production-rate-of-entropy-langevin}
\end{prop}

For the potential $V$ in \eqref{special-hamiltonian}, we make the assumption that there exist constants $L_1, L_2, L \ge~0$, such that 
\begin{align}
  \Big\|\frac{\partial V}{\partial s}(\cdot, s)\Big\|_\infty \le L_1\,,\quad 
  \Big\|\frac{\partial \nabla V}{\partial s}(\cdot, s)\Big\|_\infty \le L_2\,, \quad 
  \sup_{q\in \mathbb{R}^n} \|\nabla^2 V(q, s)\|_{2} \le L\,, \quad \forall s\in [0,T]\,,
  \label{langevin-assump-1}
\end{align}
 where $\|\cdot\|_\infty$, $\|\cdot\|_2$ are the supremum norm and the
 $2$-matrix norm defined in Section~\ref{sec-intro}, respectively. Let $a, b,
 c >0$ be positive constants such that the (two by two) matrices
\begin{align}
  S = 
  \begin{pmatrix}
    a & b \\
    b & c 
  \end{pmatrix}\,,
  \quad
  \widetilde{S} = 
  \begin{pmatrix}
    \xi\left(\frac{1}{\beta} + 2a\right)-2b(1+L) & -(a + b\xi + cL) \\
    -(a + b\xi + cL) & 2b-c 
  \end{pmatrix}
  \label{S-1-2}
\end{align}
where $L\ge 0$ is the upper bound in \eqref{langevin-assump-1}, are positive
semidefinite and positive definite, respectively. Denoting by $\lambda_i$ and
$\widetilde{\lambda}_i$, where $i=1,2$, the eigenvalues of $S$ and
$\widetilde{S}$, respectively, we have (the two eigenvalues are different
since the off-diagonal entries are nonzero)
  \begin{equation}
    0 \le \lambda_1 < \lambda_2,~ \mbox{and} \quad 0 < \widetilde{\lambda}_1 < \widetilde{\lambda}_2\,. 
    \label{s-eigenvalues}
  \end{equation}

We are ready to state the main result concerning the upper bound of the
relative entropy~\eqref{entropy-sr-langevin} in the case \eqref{special-hamiltonian}.
\begin{thm}
  Let $\mathcal{R}^{\mathrm{Lan}}(s)$ be the relative entropy \eqref{entropy-sr-langevin}. Consider the case \eqref{special-hamiltonian},
  where $\xi>0$ is a constant and $V$ is $C^2$-smooth, such that \eqref{langevin-assump-1} holds for some constants $L_1, L_2, L \ge 0$.
  Assume that the marginal measure $\nu_s^\infty$ \eqref{mu-s-langevin-q-margin} satisfies the logarithmic Sobolev inequality with constant $\kappa>0$, for all $s\in [0, T]$.
Let $a, b, c >0$ be constants, which satisfy $2b>c$ and $ac \ge b^2$, such
  that the matrices $S$ and $\widetilde{S}$ \eqref{S-1-2} are positive
  semidefinite and positive definite, respectively.
Define 
\begin{equation}
  \omega = \frac{\widetilde{\lambda}_1}{2} \Big(\frac{1}{2\min\{\kappa, \beta\}} + \lambda_2\Big)^{-1}\,,
  \label{choice-of-omega}
\end{equation}
  where $\widetilde{\lambda}_1$ and $\lambda_2$ are the eigenvalues in
  \eqref{s-eigenvalues}. Then, we have 
  \begin{align}
    \mathcal{R}^{\mathrm{Lan}}(s) \le \mathcal{E}(s) \le \mathcal{E}(0)
    \mathrm{e}^{-\omega s} + \frac{\beta^2L_1^2}{2\omega^2} + \frac{\beta^2}{\omega}\left(c+
    \frac{b}{2}\right)L_2^2\,, \quad \forall~s\in[0,T]\,,
  \label{upper-bound-thm}
  \end{align}
  where 
  \begin{equation}
\begin{aligned}
  \mathcal{E}(s) = \mathcal{R}^{\mathrm{Lan}}(s) +  \int_{\mathbb{R}^n \times
  \mathbb{R}^n} 
  \bigg[&\,a \left|\nabla_p\left(\ln\frac{d\pi_s}{d\pi^\infty_s}\right)\right|^2  +
  2b
  \nabla_p\Big(\ln\frac{d\pi_s}{d\pi^\infty_s}\Big)\cdot
  \nabla_q\Big(\ln\frac{d\pi_s}{d\pi^\infty_s}\Big)
   \\
  &+ c\left|\nabla_q\left(\ln\frac{d\pi_s}{d\pi^\infty_s}\right)\right|^2\bigg]
  \,d\pi_s \,.
\end{aligned}
  \label{modified-ent}
  \end{equation}
  \label{thm-entropy-langevin}
\end{thm}

The quantity \eqref{modified-ent} has been introduced in
Ref.~\onlinecite[Section 6]{villani2009hypocoercivity}, where a general hypocoercivity
theory was developed. The proof of Theorem~\ref{thm-entropy-langevin} is in
fact adapted from Ref.~\onlinecite[Theorem 28, Section
6]{villani2009hypocoercivity} and is given in Appendix~\ref{app-sec-entropy-estimate-langevin} due to its technicality. 

      Concerning the estimate \eqref{upper-bound-thm}, since $L_1$, $L_2$ are
      the upper bounds in \eqref{langevin-assump-1},  the estimate \eqref{upper-bound-thm}
      implies that the relative entropy $\mathcal{R}^{\mathrm{Lan}}(s)$ \eqref{entropy-sr-langevin} is small when the potential $V$ (and its gradient) varies slowly in time. In particular, when
      $L_1=L_2=0$, it recovers the exponential entropy decay for Langevin
      dynamics under a time-independent potential $V(q)$ (see
      Ref.~\onlinecite[Theorem 28, Section 6]{villani2009hypocoercivity}). Let
      us point out that it is possible to alleviate the boundedness condition
      on the Hessian of $V$ in \eqref{langevin-assump-1}, by adapting the
      approach developed in the recent work Ref.~\onlinecite{CATTIAUX2019108288} to the time-inhomogeneous setting.

We conclude this section with the following remarks on Theorem~\ref{thm-entropy-langevin}.
\begin{remark}

    Instead of \eqref{langevin-assump-1}, one can work with the assumption
      \begin{equation} 
	\Big\|\frac{\partial V}{\partial s}(\cdot, s)\Big\|_\infty \le L_1(s),\quad
	\Big\|\frac{\partial \nabla V}{\partial s}(\cdot, s)\Big\|_\infty \le
	L_2(s), \quad \sup_{q\in \mathbb{R}^n} \left\|\nabla^2 V\right(q, s)\|_{2} \le L\,,
	\label{langevin-assump-1-generationa-in-remark}
      \end{equation}
      for all $s\in [0,\infty)$, where $L>0$ is a constant, both $L_1(s)$ and $L_2(s)$ are continuous functions taking nonnegative values. Also assume
       that $\nu_s^\infty$ \eqref{mu-s-langevin-q-margin} satisfies
       logarithmic Sobolev inequality with time-dependent constant
       $\kappa(s)$, for all $s\in [0, \infty)$. Then, the same proof of Theorem~\ref{thm-entropy-langevin} actually gives 
  \begin{equation}
      \mathcal{E}(s) \le \mathrm{e}^{-\int_0^s \omega(u)\,du}\mathcal{E}(0) + 
    \frac{\beta^2}{2}\int_0^s \frac{L_1^2(u)}{\omega(u)}\,\mathrm{e}^{-\int_u^s \omega(u')\,du'}\,du 
       + 
\beta^2 \left(c+ \frac{b}{2}\right) \int_0^s L_2^2(u)\mathrm{e}^{-\int_u^s \omega(u')\,du'}\,du \,,
    \label{entropy-bound-time-depend-l1l2}
  \end{equation}
  where (cf. \eqref{choice-of-omega}) $\omega(s) = \frac{\widetilde{\lambda}_1}{2} \Big(\frac{1}{2\min\{\kappa(s), \beta\}} +
  \lambda_2\Big)^{-1}$. 
In particular, \eqref{entropy-bound-time-depend-l1l2} implies that $\lim\limits_{s\rightarrow +\infty}
  \mathcal{R}^{\mathrm{Lan}}(s) = 0$, when both $\lim\limits_{s\rightarrow +\infty} L_1(s) = 
  \lim\limits_{s\rightarrow +\infty} L_2(s) = 0$ and there exists a constant $\bar{\kappa}>0$
  such that $\kappa(s)\ge \bar{\kappa}$ for all $s \ge 0$. 
      \label{rmk-assumption-langevin}
\end{remark}
\begin{remark}
  Concerning sufficient conditions (on the constants $a,b$ and $c$) for the
  positive (semi)definiteness of $S$ and $\widetilde{S}$ in \eqref{S-1-2}, it is
  clear that $S$ is positive semidefinite when $a,b,c > 0$ and $ac \ge b^2$. Under
  the additional condition $2b>c$, the matrix $\widetilde{S}$ is positive definite as long as $a,b$ and $c$ are small enough. When $a=b=c$, 
  in particular, the integral in \eqref{modified-ent} recovers the degenerate Fisher
  information considered in Refs.~\onlinecite{Letizia-olla2017,Iacobucci2019},
  which has the advantage that only one (instead of three) constant has to be chosen in \eqref{S-1-2}.

  Next, we discuss the scaling of the constants $a,b,c$ and $\omega$ in two
  asymptotic regimes of $\xi$ based on the explicit expressions in \eqref{S-1-2}. When $\xi\rightarrow 0$, 
it is necessary that the constants $a,b$ and $c$ scale as $\mathcal{O}(\xi)$, in order for $\widetilde{S}$ to
  be positive definite. In this case, the eigenvalues in \eqref{s-eigenvalues}
  and the constant $\omega$ \eqref{choice-of-omega} are $\mathcal{O}(\xi)$.
  Similarly, when $\xi\rightarrow +\infty$, it is necessary that the constants
  $b, c$ and $\omega$ scale as $\mathcal{O}(\xi^{-1})$ (although the constant $a$ can be $\mathcal{O}(1)$). Note that the scaling of $\omega$ in these two regimes is the same as the scaling of the convergence rates in the previous work;~\cite{hairer-Pavliotis-langevin2008,exponential-rate-schmeister-2012,Iacobucci2019}
  however, in the time-inhomogeneous case the scaling of the last two terms
  (i.e.\ the terms involving $L_1$ and $L_2$) in \eqref{upper-bound-thm} needs to be
  taken into account as well, when analyzing the dependence of the upper bound \eqref{upper-bound-thm} on $\xi$.
  \label{rmk-omega-on-xi}
\end{remark}
\subsection{Connection between time reversal of reverse process and optimally controlled forward process}
\label{subsec-crooks-control-langevin}
In this section we study the connection between an optimally controlled
Langevin process and the time reversal of the reverse Langevin process~\eqref{langevin-eqn-backward}. 
Let us first introduce the stochastic optimal control problem
\begin{align}
   \mathcal{U}(q,p,t) = \inf_{(u_s)_{s\in [t,T]}} \mathbf{E}\left(\mathcal{W}^u_{(t,T)} + \frac{1}{4} \int_t^T
  |u_s|^2\,ds\,\middle|\,q^u(t)=q, p^u(t)=p\right)\,, 
  \label{opt-control-problem-U-langevin}
\end{align}
with $(q,p,t) \in \mathbb{R}^n\times \mathbb{R}^n\times [0, T]$, of the controlled process
\begin{align}
  \begin{split}
    dq^u(s) =& \nabla_p H(q^u(s),p^u(s), s)\,ds \\
    dp^u(s) =& -\nabla_q H(q^u(s),p^u(s), s)\,ds - \gamma(q^u(s),s)
    \nabla_pH(q^u(s),p^u(s), s)\,ds \\
    & + \sigma(q^u(s),s) u_s\,ds + \sqrt{2\beta^{-1}}
    \sigma(q^u(s),s)\,dw(s)\,,
  \end{split}
  \label{langevin-eqn-forward-u}
\end{align}
where $u_s \in \mathbb{R}^m$, $0 \le s \le T$, is the control force, the
infimum is over all processes $(u_s)_{s\in[0,T]}$ 
(which can depend on the past history of $(q^u(r), p^u(r))$ for $0 \le r\le s$ when control $u_s$ is chosen)
such that \eqref{langevin-eqn-forward-u} has a strong solution and 
\begin{align}
  \mathcal{W}^u_{(t,T)} = \int_{t}^{T} \frac{\partial H}{\partial
  s}\big(q^u(s), p^u(s),s\big)\,ds\,, \quad t \in [0,T]\,.
  \label{work-w-langevin-u}
\end{align}

Assume that the value function $\mathcal{U}$
\eqref{opt-control-problem-U-langevin} is a classical solution to the
Hamilton-Jacobi-Bellman equation \eqref{hjb-eqn-langevin} in
Appendix~\ref{app-subsec-control-jarzynski-langevin} (see Remark~\ref{rmk-positivity-existence-of-u-langevin} below). Then, it is known that the optimal control $u^*$ of \eqref{opt-control-problem-U-langevin}--\eqref{langevin-eqn-forward-u} exists (such that the infimum in \eqref{opt-control-problem-U-langevin} is achieved), and is given by
\begin{align}
  u_s^* = -2\sigma^\top(q,s) \nabla_p\,\mathcal{U}(q,p,s)\,,\quad s\in
  [0,T]\,,
  \label{us-opt-change-of-measure-langevin}
\end{align}
when the state of the system is at $(q,p)\in \mathbb{R}^n\times \mathbb{R}^n$ at time $s$. 
We also introduce the probability measure $\pi_0^*$ on $\mathbb{R}^n\times
\mathbb{R}^n$, which is defined as
\begin{align}
  d\pi_0^*(q,p) = \frac{1}{\mathcal{Z}(T)} \,
  \mathbf{E} \left(\mathrm{e}^{-\beta \mathcal{W}}\,\middle|\, q(0)=q, p(0)=p\right) \mathrm{e}^{-\beta H(q,p,0)}\,dqdp \,,
  \label{opt-mu0-langevin}
\end{align}
where $\mathcal{Z}(T)$ is defined in \eqref{mu-s-langevin} and (cf.\ \eqref{work-w-langevin-u})
\begin{align}
  \mathcal{W} = \int_{0}^{T} \frac{\partial H}{\partial s}\big(q(s), p(s),s\big)\,ds\,, 
  \label{work-w-langevin}
\end{align}
 and $\mathbf{E}(\cdot\,|\,q(0)=q,p(0)=p)$ denotes the path ensemble average of
\eqref{langevin-eqn-forward} starting from $(q,p)$ at $s=0$. (Note that
Jarzynski's equality \eqref{jarzynski-langevin} in
Appendix~\ref{app-subsec-control-jarzynski-langevin} implies that the integration of $\pi_0^*$ over $\mathbb{R}^n\times \mathbb{R}^n$ is indeed one.) 
It turns out that the optimal control $u^*$ and the probability measure $\pi_0^*$ provide the
optimal importance sampling Monte Carlo estimators (for free energy calculations) based on the Jarzynski's equality for Langevin
dynamics. We refer to Appendix~\ref{app-subsec-control-jarzynski-langevin} for
further motivations
  of \eqref{opt-control-problem-U-langevin}--\eqref{langevin-eqn-forward-u}. 

Our main result is stated below, which relates the optimally controlled
Langevin process \eqref{langevin-eqn-forward-u} with $u=u^*$ in
\eqref{us-opt-change-of-measure-langevin} to the time reversal of the
reverse Langevin process~\eqref{langevin-eqn-backward}. 
\begin{thm}
  Let $(q^R(s), p^R(s))_{s\in[0,T]}$ be the reverse process \eqref{langevin-eqn-backward}
  starting from the initial distribution $\pi_T^\infty$ \eqref{mu-s-langevin} at $s=0$.
Assume that the probability distribution of $(q^R(s), p^R(s))\in
  \mathbb{R}^n\times \mathbb{R}^n$ has a positive
$C^\infty$-smooth density with respect to Lebesgue measure for all $s \in [0,T]$.
  Define $(q^{R,-}(s), p^{R,-}(s))=(q^R(T-s), p^R(T-s))$ for $s \in [0,T]$.
  Denote by $(q^{u^*}(s), p^{u^*}(s))_{s\in [0,T]}$ the controlled process
  \eqref{langevin-eqn-forward-u} under $u^*$ \eqref{us-opt-change-of-measure-langevin} starting from
  the distribution $\pi_0^*$ \eqref{opt-mu0-langevin}.
  Then, $(q^{R,-}(s), p^{R,-}(s))_{s\in [0,T]}$ and $(q^{u^*}(s),
  p^{u^*}(s))_{s\in [0,T]}$ have the
  same law on the path space $C([0,T], \mathbb{R}^n\times \mathbb{R}^n)$.
  \label{thm-connection-langevin}
\end{thm}
\begin{proof}
  We sketch the proof since it is similar to the proof of Theorem~\ref{thm-connection-brownian}. 

  Let $\pi^R_s$ be the probability distribution of the state $(q^R(s), p^R(s))$ at $s \in [0,T]$.
  Denote by $\varrho^R(q,p,s)$ the probability density of $\pi^R_s$ with respect to Lebesgue measure.
  Similar to \eqref{fokker-planck-langevin}, $\varrho^R$ satisfies the Fokker-Planck equation
  \begin{equation}
    \frac{\partial \varrho^R}{\partial s} = \mathrm{e}^{-\beta
    H(\cdot,\cdot,T-s)} (\mathcal{Q}^R_s)^*
    \big(\mathrm{e}^{\beta H(\cdot,\cdot,T-s)} \varrho^R\big) \,,\quad s \in [0,T]\,,
  \label{fokker-planck-langevin-reverse}
  \end{equation}
  where $(\mathcal{Q}^R_s)^*$ is the adjoint generator of $\mathcal{Q}^R_s$ \eqref{l-s-generator-langevin-backward} with respect to the probability measure $\pi_{T-s}^\infty$ given in \eqref{mu-s-langevin}. 
  Recall the reverse process \eqref{langevin-eqn-backward}, which we rewrite as
\begin{align}
  \begin{split}
    dq^R(s) =& -\nabla_p \widehat{H}(q^R(s),p^R(s), s)\,ds \\
    dp^R(s) =& \nabla_q \widehat{H}(q^R(s),p^R(s), s)\,ds -\widehat{\gamma}(q^R(s),s)
    \nabla_p \widehat{H}(q^R(s), p^R(s), s)\,ds \\
    & + \sqrt{2\beta^{-1}} \widehat{\sigma}(q^R(s),s)\,dw(s)\,,
  \end{split}
  \label{langevin-eqn-reverse-repeat}
\end{align}
  where we have defined, for all $(q,p,s) \in \mathbb{R}^n\times \mathbb{R}^n\times [0,T]$,
  \begin{equation}
    \begin{aligned}
    & \widehat{H}(q,p,s) = H(q,p,T-s), \\
    & \widehat{\sigma}(q,s) =
      \sigma(q,T-s),\quad
      \widehat{\gamma}(q,s)=(\widehat{\sigma}\widehat{\sigma}^\top)(q,s) = \gamma(q,T-s)\,.
    \end{aligned}
    \label{h-sigma-gamma-hat}
  \end{equation}
  Since $\varrho^R$ is both $C^\infty$-smooth and positive, the result
  in Ref.~\onlinecite{haussmann1986} asserts that the time reversal
  $(q^{R,-}(s),p^{R,-}(s))_{s\in [0,T]} = (q^R(T-s),p^R(T-s))_{s\in [0,T]}$ is again a
  diffusion process which satisfies the SDE 
\begin{align}
  \begin{split}
    dq^{R,-}(s) =& \nabla_p \widehat{H}(q^{R,-}(s),p^{R,-}(s), T-s)\,ds \\
    dp^{R,-}(s) =& b^{-}(q^{R,-}(s), p^{R,-}(s),T-s)\,ds +
    \sqrt{2\beta^{-1}}
    \widehat{\sigma}\big(q^{R,-}(s),T-s\big)\,dw(s)
  \end{split}
  \label{langevin-eqn-reverse-reversal-1}
\end{align}
for $s \in [0,T]$, with the initial distribution $\pi^R_T$, where the drift is, for all $(q,p,s) \in \mathbb{R}^n\times \mathbb{R}^n\times [0,T]$,
\begin{align}
    b^{-}(q,p,s) =& \Big(-\nabla_q \widehat{H} + \widehat{\gamma}\nabla_p \widehat{H} +
    \frac{2}{\beta \varrho^R}
    \nabla_p\cdot(\varrho^R\,\widehat{\gamma})\Big)(q,p, s)\,. 
  \label{b-minus-langevin}
\end{align}
  Substituting \eqref{h-sigma-gamma-hat} in \eqref{b-minus-langevin}, we find
\begin{align}
  \begin{split}
    b^{-}(q,p,T-s) =& \Big(-\nabla_q \widehat{H} + \widehat{\gamma}\nabla_p \widehat{H} +
    \frac{2}{\beta \varrho^R}
    \nabla_p\cdot(\varrho^R\,\widehat{\gamma})\Big)(q,p, T-s) \\
=& \Big(-\nabla_q H + \gamma\nabla_p H + \frac{2}{\beta \varrho^{R,-}} \nabla_p\cdot(\varrho^{R,-}\,\gamma)\Big)(q,p, s) \\
=& \Big(-\nabla_q H + \gamma\nabla_p H +
    \frac{2}{\beta \varrho^{R,-}} \gamma\nabla_p\,\varrho^{R,-}\Big)(q,p, s) \\
    =& \Big(-\nabla_q H - \gamma\nabla_p H + \frac{2}{\beta}
    \gamma\nabla_p\ln(\mathrm{e}^{\beta H}\varrho^{R,-})\Big)(q,p, s) \,,
  \end{split}
  \label{b-minus-langevin-1}
\end{align}
where we have used the notation $\varrho^{R,-}(q,p,s)= \varrho^{R}(q,p,T-s)$
and the fact that $\gamma$ is independent of $p$.
Using \eqref{h-sigma-gamma-hat} and \eqref{b-minus-langevin-1},
we can rewrite \eqref{langevin-eqn-reverse-reversal-1} more explicitly as 
\begin{align}
  \begin{split}
    dq^{R,-}(s) =& \nabla_p H(q^{R,-}(s),p^{R,-}(s), s)\,ds \\
    dp^{R,-}(s) =& -\nabla_q H(q^{R,-}(s),p^{R,-}(s), s)\,ds - \gamma(q^{R,-}(s),s\big)\nabla_p
    H(q^{R,-}(s),p^{R,-}(s), s)\,ds \\
    & + \frac{2}{\beta}
    \Big(\gamma\nabla_p\ln(\mathrm{e}^{\beta H}\varrho^{R,-})\Big)(q^{R,-}(s),p^{R,-}(s), s)\,ds \\
    & + \sqrt{2\beta^{-1}} \sigma\big(q^{R,-}(s),s\big)\,dw(s)\,.
  \end{split}
  \label{langevin-eqn-reverse-reversal-2}
\end{align}
Define 
\begin{equation}
g(q,p,s) =\mathrm{e}^{\beta H}(q,p,s) \varrho^{R}(q,p,T-s) \mathcal{Z}(T)\,,
  \label{def-g-by-varrho-langevin}
\end{equation} 
where $\mathcal{Z}(T)$ is defined in \eqref{mu-s-langevin}.
Similar to the proof of Theorem~\ref{thm-connection-brownian} (see
\eqref{def-g} and \eqref{pde-for-g}), one can again show that 
$g$ and the value function $\mathcal{U}$ \eqref{opt-control-problem-U-langevin} are related by
\begin{equation}
\mathcal{U} = -\beta^{-1}\ln g\,.
  \label{log-trans-u-g-langevin}
\end{equation}
We refer to \eqref{g-pde-langevin}--\eqref{hjb-eqn-langevin} in Appendix~\ref{app-subsec-control-jarzynski-langevin} for details.
Combining \eqref{langevin-eqn-forward-u},
\eqref{us-opt-change-of-measure-langevin},
\eqref{langevin-eqn-reverse-reversal-2} and \eqref{log-trans-u-g-langevin}, 
we see that both the processes $(q^{R,-}(s), p^{R,-}(s))_{s\in [0,T]}$ and
$(q^{u^*}(s), p^{u^*}(s))_{s\in [0,T]}$ satisfy the same SDE.

Using Feynman-Kac formula, the PDE satisfied by $g$ (see \eqref{fun-g-langevin}--\eqref{g-pde-langevin} in
Appendix~\ref{app-subsec-control-jarzynski-langevin}), 
as well as $\mathcal{W}$ in \eqref{work-w-langevin}, we can compute the
density, for all $(q,p,s) \in \mathbb{R}^n\times \mathbb{R}^n\times [0,T]$,
\begin{equation}
  \varrho^{R}(q,p,s) 
  = \frac{\mathrm{e}^{-\beta H}(q,p,T-s)}{\mathcal{Z}(T)} 
  \mathbf{E} \left( \mathrm{e}^{-\beta \int_{T-s}^{T} \frac{\partial
  H}{\partial s}(q(t), p(t),t)\,dt\,}\,\middle|\, q(T-s)=q, p(T-s)=p \right) \,,
  \label{varrho-represented-by-feynmann-kac-langevin}
\end{equation}
from which we conclude that the distribution $\pi^R_T$ is the same as the distribution $\pi_0^*$ \eqref{opt-mu0-langevin}. 
To summarize, we have shown that both $(q^{R,-}(s), p^{R,-}(s))_{s\in [0,T]}$ and $(q^{u^*}(s), p^{u^*}(s))_{s\in [0,T]}$
satisfy the same SDE with the same initial distribution $\pi_0^*$.
Therefore, they obey the same law on the path space.
\end{proof}

\begin{remark}
  We refer to Remark~\ref{rmk-positivity-langevin-gaussian-case} in
  Section~\ref{subsec-forward-backward-langevin} for the smoothness of the
  density $\varrho^{R}$, while the positivity of $\varrho^{R}$ follows directly
  from~\eqref{varrho-represented-by-feynmann-kac-langevin}.
  Based on the smoothness and the positivity of $\varrho^{R}$, one obtains 
  from \eqref{def-g-by-varrho-langevin} and \eqref{log-trans-u-g-langevin}
  that the value function $\mathcal{U}$ \eqref{opt-control-problem-U-langevin}
  is a classical solution to the Hamilton-Jacobi-Bellman equation~\eqref{hjb-eqn-langevin} in Appendix~\ref{app-subsec-control-jarzynski-langevin}. 
  \label{rmk-positivity-existence-of-u-langevin}
\end{remark}

\section*{Acknowledgements}
The author thanks Carsten Hartmann for various discussions on optimal control
and change of measures in the context of diffusion processes. The author also
benefited from fruitful discussions with Gabriel Stoltz on hypocoercivity of Langevin dynamics. This work is funded by the Deutsche Forschungsgemeinschaft (DFG, German
Research Foundation) under Germany's Excellence Strategy --- The Berlin
Mathematics Research Center MATH+ (EXC-2046/1, project ID: 390685689). 

\section*{Data Availability}
Data sharing is not applicable to this article as no new data were created or analyzed in this
study.
\appendix
\section{Optimal control problems related to Jarzynski's identity}
\label{app-sec-control-jarzynski}
In this section, we discuss briefly the relevance of both the optimal control problem
\eqref{opt-control-problem-U}--\eqref{dynamics-1-u} in Section~\ref{subsec-crooks-control} and
the optimal control problem \eqref{opt-control-problem-U-langevin}--\eqref{langevin-eqn-forward-u} 
in Section~\ref{subsec-crooks-control-langevin} to free energy calculations based on Jarzynski's identity. We refer to Ref.~\onlinecite{non-equilibrium-2018} for detailed calculations.
\subsection{Brownian dynamics}
\label{app-subsec-control-jarzynski-brownian}
We consider the forward process \eqref{dynamics-1-q-vector} and recall the notation in Section~\ref{subsec-notations-brownian}.
Jarzynski's identity states that~\cite{jarzynski1997-master,jarzynski1997}
\begin{align}
\Delta F =  F(T) - F(0) = -\beta^{-1} \ln \mathbf{E} \left(\mathrm{e}^{-\beta
  W}\,\middle|\,x(0)\sim \nu^\infty_0\right)\,,
  \label{jarzynski}
\end{align}
where $F$ is the free energy \eqref{free-energy-f},
$\mathbf{E}(\cdot\,|\,x(0)\sim \nu^\infty_0)$ denotes the path ensemble average of 
\eqref{dynamics-1-q-vector} with initial distribution $\nu^\infty_0$, and $W$ is the work \eqref{work-w}.
Identity \eqref{jarzynski} provides a way of computing the free energy difference
$\Delta F$ by sampling nonequilibrium trajectories. See Refs.~\onlinecite{bias-error-2003,jarzynski-rare2006,optimal-estimator-minh-2009} for
previous studies. In the following, we recall some analysis of
\eqref{jarzynski} in Ref.~\onlinecite{non-equilibrium-2018} using change of measures, and show
how the optimal control problem \eqref{opt-control-problem-U}--\eqref{dynamics-1-u} arises.

Let $\bar{\nu}_0$ be a probability measure on $\mathbb{R}^n$ that is absolutely
continuous with respect to the Lebesgue measure, and $u_s\in \mathbb{R}^m$, $0
\le s \le T$, is a feedback control force such that the
Novikov's condition is satisfied (see Ref.~\onlinecite[Chapter 8, Section 6]{oksendalSDE}). Applying change of measures to \eqref{jarzynski}, we see that the free energy difference can also be estimated using
\begin{align}
  \Delta F = -\beta^{-1}\ln \mathbf{E} \left(\mathrm{e}^{-\beta W} \frac{d\mathbf{P}}{d
  \mathbf{P}^u_{\bar{\nu}_0}}\,\middle|\,x^u(0)\sim \bar{\nu}_0 \right)
  \label{jarzynski-ip-u}
\end{align}
where $\mathbf{E}(\cdot\,|\,x^u(0)\sim \bar{\nu}_0)$ denotes the path ensemble average of the controlled nonequilibrium process
\begin{align}
  \begin{split}
  d x^u(s)  =&  \Big(J - \gamma \nabla V + \frac{1}{\beta} \nabla \cdot \gamma\Big)(x^u(s), s)\,ds + \sqrt{2\beta^{-1}}
  \sigma(x^u(s),s)\,dw(s)\\
  & + \sigma(x^u(s), s)\, u_s\, ds \,,
  \end{split}
\label{dynamics-1-u-in-appendix}
\end{align}
starting from $x^{u}(0)\sim \bar{\nu}_0$, $\mathbf{P}$ and
$\mathbf{P}^u_{\bar{\nu}_0}$ are the path measures of the original process 
\eqref{dynamics-1-q-vector} and the controlled process \eqref{dynamics-1-u-in-appendix}, respectively.
The explicit expression of the likelihood ratio in \eqref{jarzynski-ip-u} is
given by Girsanov's theorem (see Ref.~\onlinecite[Chapter 8, Section 6]{oksendalSDE}).
Moreover, there exists an optimal change of measures $\mathbf{P}^{u^*}_{\nu^*_0}$, which corresponds to an optimal
initial distribution $\nu^*_0$ and an optimal control force $u^*$, such that the variance of 
the (importance sampling) Monte Carlo estimator based on \eqref{jarzynski-ip-u} equals
zero.~\cite{non-equilibrium-2018} In fact, a simple argument shows that 
$\nu^*_0$ satisfies
\begin{align}
  \frac{d\nu_0^*}{dx} 
  = \frac{\mathrm{e}^{-\beta V(x,0)}}{Z(T)} 
  \mathbf{E} \left( \mathrm{e}^{-\beta W}\,\middle|\,x(0)=x\right) = \frac{\mathrm{e}^{-\beta V(x,0)}}{Z(T)} g(x,0)  \,,
  \label{opt-mu0-in-appendix}
\end{align}
where 
\begin{align}
  g(x,s) = \mathbf{E} \left( \mathrm{e}^{-\beta \int_{s}^{T} \frac{\partial
  V}{\partial s}(x(t), t)\,dt}\,\middle|\,x(s)=x\right) \,, \quad
  \forall~ (x,s) \in \mathbb{R}^n\times [0,T]\,.
  \label{fun-g}
\end{align}
Feynman-Kac formula implies that $g$ solves the PDE
\begin{align}
  &\frac{\partial g}{\partial s} + \mathcal{L}_s g -\beta \frac{\partial V}{\partial s} g =
    0\,,\quad \mbox{and}\quad g(\cdot, T) \equiv 1\,.
    \label{g-pde}
\end{align}

The optimal control problem \eqref{opt-control-problem-U}--\eqref{dynamics-1-u}
is recovered by considering the logarithmic transformation $U=-\beta^{-1}\ln
g$ (see Ref.~\onlinecite[Example III.8.2; Chapter VI]{fleming2006}).
In fact, using the identity \eqref{lg-llogg} in Lemma~\ref{lemma-generator-property},
we can deduce from \eqref{g-pde} that $U:\mathbb{R}^n\times [0,T]\rightarrow \mathbb{R}$ satisfies the
Hamilton-Jacobi-Bellman equation 
\begin{align}
  \begin{split}
    &\frac{\partial U}{\partial s} + \min_{v\in
    \mathbb{R}^{m}}\Big\{\mathcal{L}_s U + (\sigma v)\cdot \nabla U + \frac{|v|^2}{4} + \frac{\partial V}{\partial s}\Big\} =
    0\,,\quad (x,s) \in \mathbb{R}^n \times [0,T]\,,\\
  & U(\cdot, T) = 0\,.
  \end{split}
  \label{hjb-eqn}
\end{align}
Therefore, $U$ is the value function of the stochastic optimal control problem
\eqref{opt-control-problem-U}--\eqref{dynamics-1-u} (Ref.~\onlinecite[Chapter III]{fleming2006}).
It is also known that the optimal control of \eqref{opt-control-problem-U}--\eqref{dynamics-1-u} is given by
\eqref{us-opt-change-of-measure}, which coincides with $u^*$ that leads to the optimal change of
measures \eqref{jarzynski-ip-u}. 
\subsection{Langevin dynamics}
\label{app-subsec-control-jarzynski-langevin}

Similar to the case of Brownian dynamics in Appendix~\ref{app-subsec-control-jarzynski-brownian}, Jarzynski's identity 
\begin{align}
  \Delta \mathcal{F}=  \mathcal{F}(T) - \mathcal{F}(0) = -\beta^{-1} \ln
  \mathbf{E} \left(\mathrm{e}^{-\beta \mathcal{W}}\,\middle|\,\left(q(0), p(0)\right)\sim \pi^\infty_0\right)
  \label{jarzynski-langevin}
\end{align}
holds for Langevin dynamics,~\cite{Jarzynskia2008} where $\mathcal{F}$ is the free energy \eqref{free-energy-f-langevin}, 
$\mathbf{E}(\cdot\,|\,\left(q(0), p(0)\right)\sim \pi^\infty_0)$ denotes the path ensemble average of the (forward) process \eqref{langevin-eqn-forward} starting from the initial distribution $\pi^\infty_0$ in \eqref{mu-s-langevin}, and $\mathcal{W}$ is the work \eqref{work-w-langevin}.
In the following, we explain the relation between the optimal control problem 
\eqref{opt-control-problem-U-langevin}--\eqref{langevin-eqn-forward-u} and \eqref{jarzynski-langevin}.

Applying change of measures to \eqref{jarzynski-langevin}, we have
\begin{align}
  \Delta \mathcal{F} = -\beta^{-1}\ln \mathbf{E}
  \left(\mathrm{e}^{-\beta \mathcal{W}} \frac{d\mathbf{P}}{d
  \mathbf{P}^u_{\bar{\pi}_0}}\,\middle|\,
  \left(q^u(0), p^u(0)\right)\sim \bar{\pi}_0 \right)\,,
  \label{jarzynski-ip-u-langevin}
\end{align}
where $\mathbf{P}^u_{\bar{\pi}_0}$ and $\mathbf{E}(\cdot\,|\,\left(q^u(0), p^u(0)\right)\sim \bar{\pi}_0)$ denote 
respectively the probability measure (in path space) and the path ensemble average of the controlled Langevin process
\begin{align}
  \begin{split}
    dq^u(s) =& \nabla_p H(q^u(s),p^u(s), s)\,ds \\
    dp^u(s) =& -\nabla_q H(q^u(s),p^u(s), s)\,ds - \gamma(q^u(s),s)
    \nabla_pH(q^u(s),p^u(s), s)\,ds \\
    & + \sigma(q^u(s),s) u_s\,ds + \sqrt{2\beta^{-1}} \sigma(q^u(s),s)\,dw(s)
  \end{split}
  \label{langevin-eqn-forward-u-appendix}
\end{align}
starting from the initial distribution $\bar{\pi}_0$ (which may differ from
$\pi^\infty_0$), $\mathbf{P}$ denotes the probability measure of \eqref{langevin-eqn-forward} starting from
$\pi^\infty_0$, and $u_s\in \mathbb{R}^m$, $0 \le s \le T$, is the control force.
Note that in \eqref{langevin-eqn-forward-u-appendix} the control force is only
applied to the equation of momenta~$p$. The explicit expression of the likelihood ratio in
\eqref{jarzynski-ip-u-langevin} is again given by Girsanov's theorem.
In particular, under mild conditions (such that the Hamilton-Jacobi-Bellman
equation \eqref{hjb-eqn-langevin} below has classical solution; see
Remark~\ref{rmk-positivity-existence-of-u-langevin}), there is an optimal change of measure, characterized by the optimal
initial distribution $\pi^*_0$ and the optimal control force $u^*$, such that the variance of 
the importance sampling Monte Carlo estimator based on \eqref{jarzynski-ip-u-langevin} equals
zero. A simple analysis shows that 
\begin{align}
  d\pi_0^* = \frac{1}{\mathcal{Z}(T)} g(q,p,0)\,\mathrm{e}^{-\beta H(q,p,0)}\,dqdp \,,
  \label{opt-mu0-langevin-appendix}
\end{align}
where
\begin{align}
  g(q,p,s) = \mathbf{E} \left( \mathrm{e}^{-\beta
  \int_{s}^{T} \frac{\partial H}{\partial s}(q(t), p(t),t)\,dt\,}\,\middle|\,
  q(s)=q, p(s)=p \right) \,, ~\forall~(q,p,s) \in \mathbb{R}^n\times \mathbb{R}^n \times [0,T]\,.
  \label{fun-g-langevin}
\end{align}
Applying Feynman-Kac formula, we find that the function $g$ in \eqref{fun-g-langevin} satisfies the PDE
\begin{align}
  &\frac{\partial g}{\partial s} + \mathcal{Q}_s g -\beta \frac{\partial H}{\partial s} g =
    0\,,\quad  g(\cdot, \cdot,T) = 1\,,
    \label{g-pde-langevin}
\end{align}
where $\mathcal{Q}_s$ is the generator of \eqref{langevin-eqn-forward} at time $s$.

The optimal control problem \eqref{opt-control-problem-U-langevin}--\eqref{langevin-eqn-forward-u}
 is then recovered by considering $\mathcal{U}=-\beta^{-1}\ln g$. Concretely, applying
\eqref{lg-llogg-langevin} in Lemma~\ref{lemma-generator-property-langevin},
one can derive from \eqref{g-pde-langevin} the Hamilton-Jacobi-Bellman equation
\begin{align}
  \begin{split}
    &\frac{\partial \mathcal{U}}{\partial s} + \min_{v\in
    \mathbb{R}^m}\Big\{\mathcal{Q}_s \mathcal{U} + (\sigma v)\cdot \nabla_p \mathcal{U} + \frac{|v|^2}{4} + \frac{\partial H}{\partial s}\Big\} = 0\,,\\
    & \mathcal{U}(\cdot, \cdot, T) = 0\,.
  \end{split}
  \label{hjb-eqn-langevin}
\end{align}
Therefore, $\mathcal{U}$ is the value function of the optimal control problem 
\eqref{opt-control-problem-U-langevin}--\eqref{langevin-eqn-forward-u}
(Ref.~\onlinecite[Chapter III]{fleming2006}).
The optimal control of
\eqref{opt-control-problem-U-langevin}--\eqref{langevin-eqn-forward-u} is
given by \eqref{us-opt-change-of-measure-langevin}, which coincides with $u^*$
that leads to the optimal change of measure in \eqref{jarzynski-ip-u-langevin}. 

\section{Relative entropy estimate for Langevin dynamics: Proof of Theorem~\ref{thm-entropy-langevin}}
\label{app-sec-entropy-estimate-langevin}
In this section, we prove Theorem~\ref{thm-entropy-langevin} in Section~\ref{subsec-entropy-langevin}.
The proof is based on the hypocoercivity theory (Ref.~\onlinecite[Section 6
and Section 7]{villani2009hypocoercivity}), which is a general framework for the
study of the convergence of degenerate kinetic equations towards equilibrium.
Before presenting the proof, we need to introduce some notation.

Recall that we consider the special case \eqref{special-hamiltonian}, i.e.\
  \begin{equation}
    \sigma=\sqrt{\xi} I_n\,,\quad H(q,p,s) = V(q,s) + \frac{|p|^2}{2}, \quad (q,p,s) \in \mathbb{R}^n\times
  \mathbb{R}^n \times [0, T]\,,
    \label{special-hamiltonian-repeat}
\end{equation}
where $\xi >0$. Let us define the operators
\begin{align}
  \begin{split}
  A =& \nabla_p\,,\quad A_i = \frac{\partial}{\partial p_i}\,,\quad  B_s = p \cdot \nabla_q - \nabla_q V \cdot \nabla_p\,,\\
    C =& \nabla_q\,,\quad C_i = [A_i, B_s] = A_iB_s-B_sA_i = \frac{\partial}{\partial q_i}\,, \quad 1 \le
    i \le n\,, 
  \end{split}
  \label{def-a-b-c}
\end{align}
where $s\in [0,T]$, $[A_i, B_s]$ denotes the commutator of $A_i$ and $B_s$, and we have used the fact that $V(q,s)$ is independent of $p$ to derive the last equality.  
We denote by $A_i^*$ and $B^*_s$ the adjoint operators
of $A_i$ and $B_s$ in $L^2(\mathbb{R}^n\times \mathbb{R}^n, \pi_{s}^{\infty})$ 
(with respect to the weighted inner product defined by $\pi_s^\infty$
\eqref{mu-s-langevin}), respectively (the dependence on $s$ is
omitted in the notation $A_i^*$, since $A_i^*$ turns out to be time-independent; see \eqref{a-b-c-relation}).
One can verify that~(see Ref.~\onlinecite[Section 7]{villani2009hypocoercivity})
\begin{align}
  \begin{split}
    & A_i^* = -\frac{\partial}{\partial p_i} + \beta p_i\,, \quad 
  A_i^*A_i = - \frac{\partial^2}{\partial p_i^2} + \beta p_i \frac{\partial}{\partial p_i}\,,
    \\
    & B_s^* = -B_s\,,\\
    & [C_i, B_s] = - \sum_{l=1}^n\frac{\partial^2 V}{\partial q_i\partial q_l}
  \frac{\partial}{\partial p_l}\,, \quad [A_i, A_j^*] = \beta
  \delta_{ij}\,,\quad [A_i, C_j] = 0\,, \quad [C_i, A^*_j] = 0\,,
  \end{split}
  \label{a-b-c-relation}
\end{align}
for $1 \le i,j \le n$, as well as 
\begin{align}
  \mathcal{Q}_s = - \frac{\xi}{\beta} \sum_{i=1}^n A^*_i A_i + B_s\,, \quad \mathcal{Q}^R_{T-s} =
  \mathcal{Q}_s^* = - \frac{\xi}{\beta} \sum_{i=1}^n A^*_i A_i - B_s\,,
  \label{l-by-a-b}
\end{align}
where $\mathcal{Q}_s, \mathcal{Q}_s^R$ 
are defined in \eqref{l-s-generator-langevin-forward} and
\eqref{l-s-generator-langevin-backward} (in the case of \eqref{special-hamiltonian-repeat}), respectively.
The identities in 
\eqref{a-b-c-relation} and \eqref{l-by-a-b} allow us to interchange the order of
two first-order differential operators (see the proof of Lemma~\ref{lemma-entropy-estimate-langevin-1} below).
To simplify notation, we introduce functions $h,u: \mathbb{R}^n\times
\mathbb{R}^{n}\times[0,T] \rightarrow \mathbb{R}$, given by
\begin{equation}
  h= \frac{d\pi_s}{d\pi^\infty_s}\,, \quad\mbox{and}\quad u=\ln h,
  \label{h-u-in-appendix}
\end{equation}
where the probability measure $\pi_s$ is defined in \eqref{pi-rho-density}.
We denote by $\langle\cdot, \cdot\rangle$ the inner product of two vectors in
$\mathbb{R}^n$. Also, we
will write $\int f d\pi_s$ and $\int f d\pi_s^\infty$ for integrations over the
entire phase space $\mathbb{R}^n\times \mathbb{R}^n$. 
With these conventions, 
\eqref{lg-llogg-langevin} in Lemma~\ref{lemma-generator-property-langevin} and
\eqref{log-pde-relative-density-langevin} in
Lemma~\ref{lemma-equation-of-log-density-nu-to-mu-langevin} imply
\begin{align}
  \frac{\partial h}{\partial s} = 
   \beta\Big(\frac{\partial V}{\partial s} - \frac{d \mathcal{F}}{ds} \Big)h +
   \mathcal{Q}^*_s h\,,
\end{align}
and
\begin{align}
  \frac{\partial u}{\partial s}
  =& 
  \beta \Big(\frac{\partial V}{\partial s} - \frac{d \mathcal{F}}{ds} \Big) +
  \mathcal{Q}^*_s u+ \frac{\xi}{\beta}|Au|^2 \,,
  \label{u-pde}
\end{align}
where we have used $\frac{\partial H}{\partial s} = \frac{\partial V}{\partial
s}$ thanks to \eqref{special-hamiltonian-repeat}. 
For the quantity $\mathcal{E}(s)$ in \eqref{modified-ent}, using \eqref{def-a-b-c} and \eqref{h-u-in-appendix}, we have
\begin{align}
  \mathcal{E}(s) = \mathcal{R}^{\mathrm{Lan}}(s) + a \int  |Au|^2 \,d\pi_s + 2b\int  \langle
  Au, Cu\rangle \,d\pi_s + c\int  |Cu|^2 \,d\pi_s \,.
  \label{modified-ent-appendix}
\end{align}

The following result is essentially a special case of the more
general result Ref.~\onlinecite[Lemma $32$]{villani2009hypocoercivity}. We
present the proof since in the current setting $V$ is time-dependent and the
calculation is more transparent in the special case considered here. 
\begin{lemma}
  Let $h$ and $u$ be the functions in \eqref{h-u-in-appendix}. Then, we have the following identities.
\begin{align*}
  \frac{d}{ds} \int |Au|^2 \,d\pi_s
  =\,& - \frac{2\xi}{\beta}\sum_{i,j=1}^n\int (A_iA_ju)^2 \,d\pi_s
  -  2\xi\int  |Au|^2 \,d\pi_s - 2 \int \langle Au, Cu\rangle \,d\pi_s\,,\\
  \frac{d}{ds} \int |C u|^2 \,d\pi_s
  =\,& 2\beta\int \langle Cu, C\frac{\partial V}{\partial s}\rangle \,d\pi_s
  -\frac{2\xi}{\beta}\sum_{i,j=1}^n\int (C_j A_i u)^2\, d\pi_s \\
  &\, - 2\int \langle [C,B_s] u, C u\rangle \,d\pi_s\,,  \\
    \frac{d}{ds} \int \langle A u, C u\rangle \,d\pi_s
    =\,&   \beta \int \langle Au, C\frac{\partial V}{\partial s}\rangle
    \,d\pi_s - \frac{2\xi}{\beta}\sum_{i,j=1}^n\int (A_j A_i u)(A_i C_j u)
    \,d\pi_s \\
    &\,- \xi\int \langle A u, Cu\rangle \,d\pi_s - \int |Cu|^2 \,d\pi_s -
    \int \langle Au, [C,B_s]u\rangle \,d\pi_s \,.
\end{align*}
  \label{lemma-entropy-estimate-langevin-1}
\end{lemma}
\begin{proof}
  We only prove the first identity. The other two identities can be derived similarly, using 
  \eqref{def-a-b-c}, \eqref{a-b-c-relation} and integration by parts.

Recall that $\mathcal{Q}^*_s$ is the adjoint operator of $\mathcal{Q}_s$ in
$L^2(\mathbb{R}^n\times \mathbb{R}^n, \pi_{s}^{\infty})$ where 
  $\pi_{s}^{\infty}$ is given in \eqref{mu-s-langevin}.
  Using $d\pi_s(dqdp) = \varrho(q,p,s)\,dqdp$ (see \eqref{pi-rho-density}), \eqref{h-u-in-appendix}, and the Fokker-Planck equation \eqref{fokker-planck-langevin}, we compute 
\begin{align*}
  &\frac{d}{ds} \int |A u|^2 \,d\pi_s \notag \\
  =&\frac{d}{ds} \int |A u|^2 \varrho\,dqdp \notag \\
  =&  \int \left(\frac{\partial}{\partial s}|A u|^2\right) \,\varrho\, dqdp + \int |A u|^2
\mathrm{e}^{-\beta H}\mathcal{Q}^*_s \big(\mathrm{e}^{\beta H} \varrho\big) \,dqdp\\
  =&  2\int \big\langle Au, A\frac{\partial u}{\partial s}\big\rangle\, d\pi_s
  + \int (\mathcal{Q}_s|A u|^2)\, d\pi_s \\
  =&  2\int \big\langle Au, A\frac{\partial u}{\partial s}\big\rangle\, h\,d\pi^\infty_s
  + \int (\mathcal{Q}_s|A u|^2)\, h\,d\pi^\infty_s\,.
\end{align*}
  Using \eqref{u-pde}, \eqref{l-by-a-b}, and integration by parts formula, we get
  \begin{align*}
  &\frac{d}{ds} \int |A u|^2 \,d\pi_s \notag \\
  =&  2\int \big\langle Au, A\frac{\partial u}{\partial s}\big\rangle\, h\,d\pi^\infty_s
  + \int (\mathcal{Q}_s|A u|^2)\, h\,d\pi^\infty_s \\
  =&  2\int \langle Au, A \mathcal{Q}_s^* u \rangle h\,d\pi^\infty_s +
    \frac{2\xi}{\beta}\int \langle
  Au, A |Au|^2 \rangle h\,d\pi^\infty_s - \frac{\xi}{\beta}\int \langle A|A u|^2,
  Ah\rangle\,d\pi^\infty_s \\
    & + \int (B_s|A u|^2) h\,d\pi^\infty_s\\
    =& \left[-\frac{2\xi}{\beta}\sum_{j=1}^n \int \langle Au, AA^*_jA_j u\rangle h\,
    d\pi^\infty_s +
  \frac{\xi}{\beta}\int \langle Ah, A |Au|^2 \rangle\,d\pi^\infty_s\right] \\
    &+ \left[- 2 \int \langle Au, AB_su\rangle h\, d\pi^\infty_s + \int (B_s|A u|^2) h\,d\pi^\infty_s\right]\\
    =:&~ \mathcal{J}_1 + \mathcal{J}_2\,,
\end{align*}
  where we have used $(Au)h=(A\ln h) h = Ah$, as well as
  $A(\frac{\partial V}{\partial s}-\frac{d\mathcal{F}}{ds}) =\nabla_p\frac{\partial V}{\partial s} = 0$.

  Concerning $\mathcal{J}_1$, using integration by parts formula, 
  the identities $A_iA_j^*=A_j^*A_i+[A_i, A_j^*]=A_j^*A_i + \beta \delta_{ij}$
  (see \eqref{a-b-c-relation}) and $A_iA_j=A_jA_i$, we can derive
\begin{align}
  \mathcal{J}_1 =& -\frac{2\xi}{\beta}
  \sum_{j=1}^n \int \langle Au, AA^*_jA_j u\rangle h\,
    d\pi^\infty_s +
  \frac{\xi}{\beta}\int \langle Ah, A |Au|^2 \rangle\,d\pi^\infty_s\\
  =& -\frac{2\xi}{\beta}\sum_{i,j=1}^n\int (A_iu)(A_iA^*_jA_ju)h\,d\pi^\infty_s
  + \frac{2\xi}{\beta}\sum_{i,j=1}^n\int (A_ih)(A_iA_ju)(A_ju)\,d\pi^\infty_s\notag \\
  =& -\frac{2\xi}{\beta}\sum_{i,j=1}^n\int (A_iu)(A^*_jA_iA_ju)h\,d\pi^\infty_s
  -\frac{2\xi}{\beta}\sum_{i,j=1}^n\int (A_iu)([A_i,A^*_j]A_ju)h\,d\pi^\infty_s\notag \\
  & + \frac{2\xi}{\beta}\sum_{i,j=1}^n\int (A_ih)(A_iA_ju)(A_ju)\,d\pi^\infty_s\notag \\
  =& -\frac{2\xi}{\beta}\sum_{i,j=1}^n\int (A_iA_ju)^2h\,d\pi^\infty_s
  -2\xi\int |Au|^2h\,d\pi^\infty_s \notag \\ 
  =& -\frac{2\xi}{\beta}\sum_{i,j=1}^n\int (A_iA_ju)^2\,d\pi_s
  -2\xi\int |Au|^2\,d\pi_s\,. \label{j-1}
\end{align}

  Concerning $\mathcal{J}_2$, using $AB_s=B_sA+[A,B_s]=B_sA+C$ (see \eqref{def-a-b-c}), we compute
    \begin{equation}
      \begin{aligned}
      \mathcal{J}_2 =& 
      - 2 \int \langle Au, AB_su\rangle h\, d\pi^\infty_s + \int (B_s|A u|^2) h\,d\pi^\infty_s\\
      = &  -  2\int \langle A u, B_sAu\rangle h \,d\pi^\infty_s - 2
      \int \langle Au, [A, B_s] u\rangle h \,d\pi^\infty_s + \int (B_s|A u|^2)
	h\,d\pi^\infty_s \\
      =& - 2 \int \langle Au, Cu\rangle h \,d\pi^\infty_s \\ 
      =& - 2 \int \langle Au, Cu\rangle \,d\pi_s\,.  
    \end{aligned}
\label{j-2}
    \end{equation}
The first conclusion follows by summing up \eqref{j-1} and \eqref{j-2}. 
\end{proof}
Now we are ready to prove Theorem~\ref{thm-entropy-langevin}.
\begin{proof}[Proof of Theorem~\ref{thm-entropy-langevin}]
First, note that in the case of \eqref{special-hamiltonian} the marginal measure of $\pi^\infty_s$ 
  \eqref{mu-s-langevin} in momenta $p$ is $Z_p^{-1}\mathrm{e}^{-\beta |p|^2/2}dp$, 
      where $Z_p=\int_{\mathbb{R}^n} \mathrm{e}^{-\beta |p|^2/2}dp$, which
      satisfies the logarithmic Sobolev inequality with constant $\beta$ (see
      Ref.~\onlinecite[Example 21.3]{villani2008optimal}).
      The assumption that the spatial marginal measure
$\nu^\infty_s$ \eqref{mu-s-langevin-q-margin} of $\pi^\infty_s$ satisfies the
      logarithmic Sobolev inequality with constant $\kappa$ implies that
      $\pi^\infty_s$ itself (as a product measure) satisfies the logarithmic
      Sobolev inequality with constant $\min\{\kappa, \beta\}$ (see
      Ref.~\onlinecite[Section 5.2]{ledoux2001concentration}). Therefore,
      using the notation introduced above, we have
      \begin{equation}
	\mathcal{R}^{\mathrm{Lan}}(s) \le \frac{1}{2\min\{\kappa, \beta\}} \int 
	\left(|Au|^2 + |Cu|^2\right)	\,d\pi_s\,.
	\label{lsi-whole-measure-pi-in-proof}
      \end{equation}
      
      Since $\lambda_i$, $\widetilde{\lambda}_i$, where $i=1,2$, are the
      eigenvalues \eqref{s-eigenvalues} of the matrices $S$ and
      $\widetilde{S}$ \eqref{S-1-2}, respectively, we have 
\begin{align}
  \begin{split}
    0 \le \lambda_1(x^2 + y^2) \le &\, ax^2 + 2bxy+cy^2 \le \lambda_2(x^2 + y^2)\,, \\
    0 \le  \widetilde{\lambda}_1(x^2 + y^2) \le &\,\Big[\xi\Big(\frac{1}{\beta} +
    2a\Big)-2b(1+L)\Big]x^2 - 2(a + b\xi+cL) xy+ (2b-c)y^2 \\
     \le&\, \widetilde{\lambda}_2(x^2 + y^2)\,,
  \end{split}
  \label{s-s1-bounds}
\end{align}
for all $x,y\in \mathbb{R}$. 
  Using \eqref{lsi-whole-measure-pi-in-proof} and the first
  estimate in \eqref{s-s1-bounds}, from the expression
  \eqref{modified-ent-appendix} we find
\begin{align}
  \mathcal{R}^{\mathrm{Lan}}(s) \le \mathcal{E}(s) \le
  \Big(\frac{1}{2\min\{\kappa, \beta\}} +
  \lambda_2\Big) \int  \big(|Au|^2 + |Cu|^2\big)\,d\pi_s \,, \quad \forall s \in [0, T]\,.
  \label{r-e-bound}
\end{align}
  Applying the identities in both Proposition~\ref{prop-production-rate-of-entropy-langevin} and
  Lemma~\ref{lemma-entropy-estimate-langevin-1}, we can derive
\begin{align}
  & \frac{d}{ds} \mathcal{E}(s) \notag \\
    =   &  
     -\beta \int \frac{\partial V}{\partial
     s} d\pi_s^\infty + \beta \int \frac{\partial V}{\partial s}\, d\pi_s \notag\\
    & + 2c\beta \int \langle Cu, C\frac{\partial V}{\partial s}\rangle \,d\pi_s 
    + 2b\beta \int \langle Au, C\frac{\partial V}{\partial s}\rangle \,d\pi_s \notag\\
    & - \xi\Big(\frac{1}{\beta} + 2a\Big)\int  |Au|^2 \,d\pi_s - 2(a+b\xi)\int \langle A u, C u\rangle \,d\pi_s 
     - 2b\int |Cu|^2 \,d\pi_s \notag\\
     & -  2b\int \langle Au, [C,B_s]u\rangle \,d\pi_s - 2c\int \langle [C,B_s] u, C u\rangle \,d\pi_s \notag\\
    & -\frac{2\xi}{\beta}\sum_{i,j=1}^n
    \int \Big[a (A_iA_ju)^2 + 2b (C_j A_i u) (A_iA_ju) + c (C_j A_i
    u)^2\Big] \,d\pi_s\notag \\
    =:& \mathcal{J}_1 + \mathcal{J}_2 + \mathcal{J}_3 + \mathcal{J}_4 + \mathcal{J}_5\,, \label{exp-of-e}
\end{align}
where $\mathcal{J}_l$ denotes the terms on the $l$th line in the second
  equality above, for $1 \le l \le 5$. 
  Clearly, the first estimate in \eqref{s-s1-bounds} implies that $\mathcal{J}_5 \le 0$. In the following, we estimate $\mathcal{J}_l$ for $l=1,2,3,4$.

Recall the constants $L_1$, $L_2$ and $L$ in the assumption \eqref{langevin-assump-1}.
Similar to \eqref{term1-bounded-by-sqrt-of-r}, using
\eqref{integration-bounded-by-tv}, \eqref{langevin-assump-1}, and Csisz\'ar-Kullback-Pinsker inequality \eqref{csiszar-kullback-pinsker}, we find 
  \begin{align}
    \begin{split}
    \mathcal{J}_1 \le& \beta \left| \int \frac{\partial V}{\partial s} d\pi^{\infty}_s - 
    \int \frac{\partial V}{\partial s}\, d\pi_s\right| \\
    \le& \beta L_1
    \sqrt{2D_{KL}(\pi_s\,\|\,\pi^\infty_s)} =
    \beta L_1\sqrt{2\mathcal{R}^{\mathrm{Lan}}(s)} \le \frac{\beta^2L_1^2}{2\omega} + \omega
    \mathcal{E}(s)\,, 
      \end{split}
    \label{app-bound-1}
  \end{align}
where $\omega$ is given in \eqref{choice-of-omega} and the last inequality follows from Young's inequality, while Cauchy-Schwarz inequality and the assumption \eqref{langevin-assump-1} imply
  \begin{align}
    \begin{split}
      \mathcal{J}_2 \le & \left|2c\beta \int \langle Cu, C\frac{\partial V}{\partial s}\rangle \,d\pi_s 
    + 2b\beta \int \langle Au, C\frac{\partial V}{\partial s}\rangle
      \,d\pi_s\right| \\
    \le &
      \left(c+\frac{b}{2} \right)\beta^2 L_2^2 + c\int |Cu|^2\,d\pi_s + 2 b \int |Au|^2\,d\pi_s\,.
    \end{split}
    \label{app-bound-2}
  \end{align}
  For $\mathcal{J}_3$, it is clear that
  \begin{equation}
    \mathcal{J}_3 \le -\xi\left(\frac{1}{\beta} + 2a\right) 
 \int  |Au|^2 \,d\pi_s + 2(a+b\xi)\int  |A u|\,|C u| \,d\pi_s 
     - 2b\int |Cu|^2 \,d\pi_s \,.
    \label{app-bound-3}
  \end{equation}
  For $\mathcal{J}_4$, since $[C_i, B_s]= -\sum_{l=1}^n\frac{\partial^2 V}{\partial q_i\partial q_l}
  \frac{\partial}{\partial p_l}$ (see \eqref{a-b-c-relation}) and $\|\nabla^2 V\|_2 \le L$ by assumption \eqref{langevin-assump-1}, we have 
  \begin{equation}
    \begin{aligned}
      |\mathcal{J}_4| =\, & \left|-2b\int \langle Au, [C,B_s]u\rangle \,d\pi_s
      - 2c\int \langle [C,B_s] u, C
     u\rangle \,d\pi_s\right| \\
       \le & 2bL \int |Au|^2 d\pi_s + 2cL \int |Au|\,|Cu| d\pi_s\,.
    \end{aligned}
    \label{app-bound-4}
  \end{equation}

Substituting \eqref{app-bound-1}--\eqref{app-bound-4} into \eqref{exp-of-e}, 
 applying \eqref{s-s1-bounds} and \eqref{r-e-bound}, we find
\begin{align*}
  & \frac{d}{ds} \mathcal{E}(s) \\
    \le &
  \left[\frac{\beta^2L_1^2}{2\omega} + \left(c+ \frac{b}{2}\right)\beta^2L_2^2\right] + \omega
  \mathcal{E}(s) -\bigg\{\left[\xi\Big(\frac{1}{\beta} + 2a\Big)-2b(1+L)\right] \int |Au|^2 \,d\pi_s\\
    & -  2(a + b\xi+cL)\int |Au| |Cu| \,d\pi_s  
     + (2b-c)\int |Cu|^2 \,d\pi_s\bigg\} \notag\\
    \le& \left[\frac{\beta^2 L_1^2}{2\omega} + \left(c+ \frac{b}{2}\right)\beta^2L_2^2\right] + \omega \mathcal{E}(s)
    - \widetilde{\lambda}_1 \int  (|Au|^2+|Cu|^2) \,d\pi_s \notag\\
    \le& \left[\frac{\beta^2L_1^2}{2\omega} + \left(c+
    \frac{b}{2}\right)\beta^2L_2^2\right] 
    + \omega \mathcal{E}(s) - \widetilde{\lambda}_1
    \Big(\frac{1}{2\min\{\kappa, \beta\}} + \lambda_2\Big)^{-1} \mathcal{E}(s)\\
    =& \left[\frac{\beta^2L_1^2}{2\omega} + \left(c+ \frac{b}{2}\right)\beta^2L_2^2\right]
    - \omega \mathcal{E}(s) \,,
\end{align*}
where the final equality follows from the choice of $\omega$ in \eqref{choice-of-omega}.
The conclusion follows by integrating the inequality above (Gronwall's inequality).
\end{proof}
\bibliographystyle{abbrv}
\bibliography{reference}

\end{document}